\documentclass[notitlepage,12pt,reqno]{amsart}
\usepackage{amssymb,amsmath,amscd,xspace,amsthm,mathrsfs}
\usepackage[breaklinks=true]{hyperref}
\usepackage[usenames]{color}
\usepackage{graphicx}
\usepackage{enumerate}
\usepackage[all]{xy}

\newcommand{\N}{\ensuremath{\mathbb N}}
\newcommand{\scrs}{\ensuremath{\mathscr S}}
\newcommand{\Z}{\ensuremath{\mathbb Z}}
\newcommand{\R}{\ensuremath{\mathbb R}}

\newcommand{\Q}{\ensuremath{\mathcal Q}}
\newcommand{\qhat}{\ensuremath{\widehat{\mathcal Q}}}

\newcommand{\ptild}{\ensuremath{\widetilde{P}}}

\newcommand{\G}{\ensuremath{\mathcal G}}
\newcommand{\HH}{\ensuremath{\mathcal H}}

\newcommand{\W}{\ensuremath{\mathcal W_{[0,1]}}}
\newcommand{\WW}{\ensuremath{\mathcal W}}
\newcommand{\Wopen}[1]{\ensuremath{\mathcal W}_{(-#1,#1)}}
\newcommand{\Wclosed}[1]{\ensuremath{\mathcal W}_{[-#1,#1]}}
\newcommand{\Wlimit}[1]{\ensuremath{\mathcal W}_{\uparrow #1}}
\newcommand{\cutnorm}[1]{\ensuremath{\| #1 \|_{\rm cut}}}
\newcommand{\boxnorm}[1]{\ensuremath{\| #1 \|_\Box}}

\newcommand{\lie}[1]{\ensuremath{\mathfrak{#1}}}
\newcommand{\novert}[1]{\ensuremath{\widetilde{#1}}}
\newcommand{\comment}[1]{}

\def\img {\mathop{\rm Im}}

\def\Aut{\mathop{\rm Aut}}
\def\Surj{\mathop{\rm Surj}}
\def\Adm{\mathop{\rm Adm}}

\DeclareMathOperator{\cut}{\ensuremath{cut}\xspace}

\newtheorem{theorem}[equation]{Theorem}
\newtheorem{prop}[equation]{Proposition}

\newtheorem{cor}[equation]{Corollary}
\newtheorem{lemma}[equation]{Lemma}

\theoremstyle{definition}
\newtheorem{remark}[equation]{Remark}
\newtheorem{definition}[equation]{Definition}

\numberwithin{equation}{section}

\usepackage[letterpaper]{geometry}
\begin{document}
\title{Differential Calculus on Graphon Space}
\author[Peter Diao, Dominique Guillot, Apoorva Khare, and Bala
Rajaratnam]{Peter Diao, Dominique Guillot, Apoorva Khare, and Bala
Rajaratnam\\Stanford University}
\thanks{P.D., D.G., A.K., and B.R.~are partially supported by the
following: US Air Force Office of Scientific Research grant award
FA9550-13-1-0043, US National Science Foundation under grant DMS-0906392,
DMS-CMG 1025465, AGS-1003823, DMS-1106642, DMS-CAREER-1352656, Defense
Advanced Research Projects Agency DARPA YFA N66001-111-4131, the UPS
Foundation, SMC-DBNKY, and an NSERC postdoctoral fellowship}
\address{Department of Mathematics, Stanford University, Stanford, CA
94305, USA}
\date{\today}

\begin{abstract}
Recently, the theory of dense graph limits has received attention from
multiple disciplines including graph theory, computer science,
statistical physics, probability, statistics, and group theory. In this
paper we initiate the study of the general structure of differentiable
graphon parameters $F$.
We derive consistency conditions among the higher G\^ateaux derivatives
of $F$ when restricted to the subspace of edge weighted graphs $\WW_{\bf
p}$.
Surprisingly, these constraints are rigid enough to imply that the
multilinear functionals $\Lambda: \WW_{\bf p}^n \to \R$ satisfying the
constraints are determined by a finite set of constants indexed by
isomorphism classes of multigraphs with $n$ edges and no isolated
vertices. Using this structure theory, we explain the central role that
homomorphism densities play in the analysis of graphons, by way of a new
combinatorial interpretation of their derivatives. In particular,
homomorphism densities serve as the monomials in a polynomial algebra
that can be used to approximate differential graphon parameters as Taylor
polynomials.
These ideas are summarized by our main theorem, which asserts that
homomorphism densities $t(H,-)$ where $H$ has at most $N$ edges form a
basis for the space of smooth graphon parameters whose $(N+1)$st
derivatives vanish.
As a consequence of this theory, we also extend and derive new proofs of
linear independence of multigraph homomorphism densities, and
characterize homomorphism densities. In addition, we develop a theory of
series expansions, including Taylor's theorem for graph parameters and a
uniqueness principle for series. We use this theory to analyze questions
raised by Lov\'asz, including studying infinite quantum algebras and the
connection between right- and left-homomorphism densities. 
Our approach provides a unifying framework for differential calculus on
graphon space, thus providing further links between combinatorics and
analysis.
\end{abstract}
\maketitle

\section{Introduction}\label{S1}

The theory of dense graphs and their limits introduced in \cite{LS:2006}
has attracted much attention recently (see e.g.~\cite{AT:2010, BCL:2010,
BCLSV:2008, BCLSV:2012, CV:2011, CV:2012, ES:2012}).
It has also been observed that limit theories developed in the context of
(i) graphons, (ii) exchangeable arrays of random variables
(\cite{Hoover:1982}), and (iii) metric measure spaces (\cite[Chapter
3$\frac{1}{2}$]{Gromov:2001} and \cite{Vershik:1998}) can often be
translated into each other (see \cite{Aldous:ICM:2010, DJ:2008,
Elek:2012}).
Several questions have benefited from reformulation in this language -
see for instance \cite{Hatami:2010, Lovasz:2011, LS:2011} and
\cite[Chapter 16]{Lovasz:2013}. The monograph \cite{Lovasz:2013} covers
many aspects of this development, including topology and analysis on the
space of graph limits.
Since graphs have become a central abstraction for the modern analysis of
complex systems, the theory has also been used to address applied
questions in the study of estimable graph parameters \cite{LS:2010},
machine learning (\cite{LOGR:2012}), and statistical modelling of
networks \cite{BC:2009, DHJ:2013,RCY:2011}.
It seems that such a language was needed as much for mathematical theory
as for practical application.

The present paper begins the study of functional analysis of dense graph
limits. In order to explain our motivation and results, we first briefly
review dense graph limit theory and set some notation.
By a \textit{graphon} we mean a bounded symmetric measurable function $f:
[0,1]^2 \to [0,1]$. Recall that a finite simple labelled graph $G$ with
vertices $V = \{ 1, 2, \dots, n \}$ is identified with the graphon $f^G$,
defined as follows:
\[ f^G(x,y) = {\bf 1}_{(\lceil nx \rceil, \lceil ny \rceil) \in E} =
\begin{cases}
    1, 	&\text{if $(\lceil nx \rceil, \lceil ny \rceil)$ is an edge in
    $G$,}\\
    0,	&\text{otherwise.}\\
\end{cases} \]

One description of the topology on isomorphism classes of finite simple
graphs in dense graph limit theory is given as follows.
The space $\W$ of all graphons sits inside $\mathcal{W}$, the vector
space of bounded symmetric measurable functions $f: [0,1]^2 \to \R$. The
space $\WW$ has a seminorm called the \textit{cut norm}
\[ \cutnorm{f} := \sup_{S,T \subset [0,1]} \left| \int_{S \times T}
f(x,y)\ dx\ dy \right|\]
where the supremum is taken over all pairs of Lebesgue measurable subsets
$S,T$ of $[0,1]$.
The monoid of measure-preserving maps $\overline{S}_{[0,1]}$ acts on $\W$
by $f^\sigma(x,y) := f(\sigma(x),\sigma(y))$ for $\sigma \in
\overline{S}_{[0,1]}$. Let $S_{[0,1]}$ be the group of invertible
measure-preserving maps and define
\begin{equation*}
\delta_\Box(f,g) = \inf_{\psi \in S_{[0,1]}} \cutnorm{f - g^\psi}.
\end{equation*}

If $G,G'$ are isomorphic finite simple graphs, then $\delta_\Box(f^G,
f^{G'}) = 0$, so $\delta_\Box$ can be used to define graph convergence of
isomorphism classes of finite graphs.  

Another description of the same topology can be given in terms of
homomorphism densities.  Given two finite simple graphs $G = (V(G),
E(G))$, $H = (V(H), E(H))$, let $\hom(H,G)$ denote the number of
edge-preserving maps $: V(H) \to V(G)$. Now define the homomorphism
densities by $t(H,G) = \hom(H,G) / |V(G)|^{|V(H)|}$. A sequence of simple
graphs $(G_n)_{n=1}^\infty$ \textit{left converges} if $t(H,G_n)$
converges for every finite simple graph $H$.  
Intuitively, a sequence of graphs left converges if their edge densities,
triangle densities, etc. all converge when normalized for dense graph
sequences.

More generally, the homomorphism densities are defined for a graphon $f$
and a multigraph $H = (\{ 1, \dots, k \}, E(H))$ by
\begin{equation}\label{Ehom}
t(H,f) := \int_{[0,1]^k} \prod_{(i,j) \in E(H)} f(x_i, x_j)\ d x_1 \cdots
d x_k.
\end{equation}
Note that if $H$ is a graph with no edges, then we define $t(H,f) \equiv
1$.  Moreover, $t(H,f^G) = t(H,G)$ for finite simple graphs $H,G$ so
these two definitions are compatible.

Theorem 3.8 of \cite{BCLSV:2008} asserts the equivalence of left
convergence and convergence in the $\delta_\Box$ topology.  It also
explains that graphons serve as limits to such convergent sequences of
graphs.

\begin{theorem}[Borgs-Chayes-Lov\'asz-S\'os-Vesztergombi,
2008]\label{thm:BCLSV2008}
Let $(W_n)_{n=1}^\infty$ be a sequence of graphons in $\W$. Then the
following are equivalent:
\begin{enumerate}[(i)]
\item $t(H,W_n)$ converges for all finite simple graphs $H$;
\item $W_n$ is a Cauchy sequence in the $\delta_\Box$ metric;
\item there exists $W \in \W$ such that $t(H,W_n) \to t(H,W)$ for all
finite simple graphs $H$.
\end{enumerate}
Furthermore, $t(H,W_n) \to t(H,W)$ for all finite simple graphs $H$ for
some $W \in \W$ if and only if $\delta_\Box(W_n,W) \to 0$.
\end{theorem}

The following theorem summarizes the results about uniqueness of graphon
representations in \cite[Theorem 2.1 and Corollary 2.2]{BCL:2010} that we
rely on in this paper.

\begin{theorem}[Borgs-Chayes-Lov\'asz, 2010]
For two graphons $f, g\in \W$, the following are equivalent:
\begin{enumerate}[(i)]
\item the graphons $f$ and $g$ satisfy $\delta_\Box(f,g) = 0$;
\item the graphons $f$ and $g$ satisfy $t(H,f) = t(H,g)$ for all finite
simple graphs $H$;
\item there exists $U \in \W$ and Lebesgue measure-preserving maps $\phi,
\psi: [0,1] \to [0,1]$ such that $U^\phi = f$ and $U^\psi = g$ almost
everywhere.
\end{enumerate}
\end{theorem}

We say that two graphons $f \sim g$ are {\it weakly equivalent} if they
satisfy any of the three equivalent conditions in the above theorem.

The basic properties of the metric space $(\W/\sim, \delta_\Box)$ are as
follows:

\begin{enumerate}
\item The countable family of graphons $f^G$ associated with simple
graphs $G$ is dense in $(\W/\sim, \delta_\Box)$ (see \cite{LS:2006}).

\item As a consequence of the Weak Regularity Lemma in graph theory,
$(\W/\sim, \delta_\Box)$ is a compact metric space (see \cite{LS:2007}).
\end{enumerate}

We refer to functions $F : \W \to \R$ which factor through $\W / \sim$ as
{\it class functions}. These are also sometimes referred to in the
literature as ``graphon parameters". We say that a function $F: \W \to
\R$ is continuous if it is continuous with respect to the cut norm on
$\W$ (unless a different topology is specified).  We say that $F:\W/\sim
\to \R$ is continuous if the lift of $F$ to $\W$ is continuous.


The fundamental work on graph limits cited above builds a bridge between
combinatorics and analysis.  In this paper, we exploit the linear
structure of $\WW$ to build a theory of differentiation and further the
understanding of analysis on $\W$.  
Let $d^n F(f;g_1, \dots, g_n)$ denote the (higher order) G\^ateaux
derivative of the function $F : \W \to \R$ at $f \in \W$ along the
directions $g_1, \dots, g_n \in \W$. (See Definition \ref{Dgateaux}.)
The main result of the present paper shows that homomorphism densities
span the complete space of solutions to fundamental differential
equations on $\W$.

\begin{theorem}\label{thm:main}
Let $F: \mathcal{W}_{[0,1]} \rightarrow \R$ be a class function which is
continuous with respect to the $L^1$ norm and $(N+1)$ times G\^ateaux
differentiable for some $N \geq 0$. Then $F$ satisfies:
\begin{equation}\label{eqn:main}
d^{N+1} F(f; g_1, \dots, g_{N+1}) = 0, \qquad \forall f \in \W,\ g_1,
\dots, g_{N+1} \in \Adm(f),
\end{equation}

\noindent if and only if there exist constants $a_H$ such that
\begin{equation*}
F(f) =\sum_{H \in \HH_{\leq N}} a_H t(H, f).
\end{equation*}

\noindent Moreover, the constants $a_H$ are unique.
(Here, $\HH_{\leq N}$ denotes the set of isomorphism classes of
multigraphs with no isolated vertices and at most $N$ edges, and
$\Adm(f)$ is the set of admissible directions for computing the G\^ateaux
derivative; see Definition \ref{Dadmissible}.)
If in addition $F$ is continuous with respect to the cut-norm, then $a_H
= 0$ if $H \in \HH_{\leq N}$ is not a simple graph.
\end{theorem}

One way to understand this theorem is that it shows that the solution
space of the natural generalization of $\frac{d^n}{dx^n}F(x) \equiv 0$ to
class functions on $\W$ is spanned by homomorphism densities.  So it
confirms the intuition that homomorphism densities play the role of
monomials in the algebra of continuous class functions.

To explain why the theorem is surprising, consider removing the condition
that $F$ is a class function on $\W$. In that case, consider any
multilinear $\Lambda : \W^n \to \R$ with $n \leq N$. Then $F(f) :=
\Lambda(f,f,\dots,f)$ would be a G\^ateaux-smooth solution to
\eqref{eqn:main}. Since the ``tangent space" to $\W$ at the origin is
infinite-dimensional, the space of such solutions is also
infinite-dimensional. Theorem \ref{thm:main} shows that imposing the
symmetry condition on $F$ (i.e., assuming that $F$ is a class function)
collapses the set of solutions to a finite-dimensional space.

The proof of this theorem requires several steps which we now outline
(see beginning of Section \ref{S4} for a more detailed outline). It
begins with the observation that the differential equation
\eqref{eqn:main} has solutions $F$ that satisfy:
\begin{equation*}
F(g) = F(0) + dF(0;g) + \frac{d^2F(0;g,g)}{2!} + \dots +
\frac{d^nF(0;g,\dots, g)}{n!}.
\end{equation*}

\noindent The proof of the theorem then proceeds by first understanding
the structure of the functional $\Lambda (g_1, \dots, g_n) := d^nF(0;g_1,
\dots, g_n)$. In fact, the maps $\Lambda : \W^n \to \R$ are multilinear
functionals satisfying rigid symmetries.  We exploit these symmetries to
show that $\Lambda$ restricted to weighted graphs are determined by
constants indexed by the isomorphism classes $\HH_n$ of graphs with $n$
edges and no isolated vertices.  This surprising local structure of
derivatives of smooth class functions at the origin is developed in
Section \ref{section:cc}. Computing the constants of Section
\ref{section:cc} for homomorphism densities shows that the derivatives of
any smooth class function at zero restricted to weighted graphs can be
written uniquely in terms of the derivatives of homomorphism densities.
In order to prove this, we give a new combinatorial interpretation of
higher G\^ateaux derivatives of homomorphism densities in Section
\ref{section:mainproof}.

Although the local structure theory of derivatives is interesting in its
own right, we show that it yields rich rewards.\medskip

\noindent {\bf Taylor polynomials.}
The finite-dimensionality of the solution spaces to \eqref{eqn:main}
allows us to develop the theory of Taylor polynomials.
In particular, we prove that every smooth continuous class function $F$
has a unique Taylor expansion where homomorphism densities play the role
of monomials. We also give sufficient conditions for when this sequence
of Taylor polynomials converges to $F$ (see Theorem \ref{thm:analytic}).
\medskip

\noindent {\bf Linear independence of homomorphism densities.}
Our techniques allow us to prove the linear independence of homomorphism
densities for partially labelled multigraphs (see Theorem
\ref{thm:gen_whitney}).  Such linear independence results go back to
Whitney \cite{Whitney} for simple graph homomorphism densities (see also
\cite{ELS:1979}). The proof technique depends on a combinatorial
interpretation of the formula for derivatives of homomorphism
densities.\medskip

\noindent {\bf Partially labelled multigraphs and infinite series.}
Theorem \ref{thm:main} shows that homomorphism densities $t(H,-)$ can be
seen as monomials with degree $|E(H)|$.
We investigate infinite {\it power series} of such monomials in the
general setting of partially labelled graphs. This allows us to provide
an answer to a question of Lov\'asz about infinite quantum algebras (see
Theorem \ref{thm:alpha_injection}).\medskip

\noindent {\bf Characterizing homomorphism densities.} 
As another application of our theory, in Theorem
\ref{thm:characterization} we characterize homomorphism densities
$t(H,-)$ as the continuous maps on $\W / \sim$ which are multiplicative
with respect to tensor products.
This complements previous characterizations of $\hom(H,-)$ and
$\hom(-,H)$ (see \cite[\S 5.6]{Lovasz:2013}).
\medskip

\subsection{Organization of the paper}

The rest of this paper is organized as follows.  

In Section \ref{Sprelim} we review some basic properties of the cut-norm
on $\W$ and characterize the continuous homomorphism densities. The main
theme is how to effectively exploit the density of the finite simple
graphs in $\W$.
In Section \ref{Sdiff} we develop the general theory of differentiation
on $\W$. We explain what kind of smoothness assumptions are required, as
well as why we use the G\^ateaux derivative in favor of the Fr\'echet
derivative.

In Section \ref{S4} we prove the main theorem by investigating the
derivatives of smooth class functions in detail. In particular, we find
that the G\^ateaux derivatives $d^nF(0;g_1,\dots,g_n)$ for a continuous
class function satisfy relations that allow us to extract combinatorial
data indexed by isomorphism classes of graphs to characterize the
function.  We also give a new proof of the linear independence of
homomorphism densities using this structure theory, as well as a
combinatorial interpretation of derivatives of homomorphism densities.
Using the results of this section and the previous sections, we prove
Theorem \ref{thm:main}. As an application, we obtain an analytic
characterization of homomorphism densities $t(H,-)$.

In Section \ref{Sseries} we consider partially labelled multigraphs and
form algebras of linear combinations of them.  We define weighted
homomorphism densities for such graphs and develop a general analytic
theory of infinite series of such functions.  As an application, we
investigate whether right homomorphism densities can be expanded in terms
of left homomorphism densities (see \cite[Problem 16]{Lovasz:Open}).  We
also explain the uniqueness and existence of Taylor series of
homomorphism densities of smooth class functions and give sufficient
conditions for their convergence.  Finally, we generalize linear
independence of homomorphism densities to partially labelled multigraphs,
and explain how to construct an analytic theory of infinite quantum
algebras. The theory allows us to address another of Lov\'asz's questions
(\cite[Problem 7]{Lovasz:Open}).

\section{Preliminaries}

In this section we review the topology on $\W$ and its basic properties.
We then introduce the general notions of differentiability on $\W$.

\subsection{Continuity and homomorphism densities}\label{Sprelim}

A parallel viewpoint on the $\delta_\Box$ metric on $\W/\sim$ to the
viewpoints discussed in Section \ref{S1} is that two graphs are close in
the $\delta_\Box$-topology if finite random subgraphs of them have
similar distributions. As a consequence, continuous class functions on
graphon space are precisely the estimable (or ``testable") graph
parameters \cite[Theorem 15.1]{Lovasz:2013}. Informally, these are the
functions of isomorphism classes of graphs, that can be estimated at a
graph from a random induced subgraph.

Our goal is to study continuous functions on $(\W/ \sim, \delta_\Box)$,
and by extension, on $\W$. As we now explain, the homomorphism densities
$t(H,-)$ are fundamental amongst such continuous functions. 
Let $\scrs_t$ denote the linear span of homomorphism densities $t(H,-)$
for $H$ a simple graph. Since for any two disjoint finite simple graphs
$H_1, H_2$, one has
\begin{equation}\label{Eproduct}
t(H_1 \coprod H_2,f) = t(H_1,f) \times t(H_2,f),
\end{equation}
\noindent $\scrs_t$ is actually an algebra.
We now prove a Stone-Weierstrass-type theorem for this algebra.

\begin{theorem}[Density theorem]\label{Stone}
The linear span of homomorphism densities $\scrs_t$ is dense in
$C(\W/\sim, \delta_\Box)$, the space of continuous functions on
$\W/\sim$, under the topology of uniform convergence.  
\end{theorem}

\begin{proof}
The space of functions $\scrs_t$ is an algebra of functions which
contains the constant function.  The space $\W/\sim$ is compact
(\cite{LS:2007}) and Hausdorff and so it suffices to check that the set
of functions separates points \cite[Theorem 7.32]{Rudin:1964}.  We need
only recall that $f \sim g$ if and only if $t(H,f) = t(H,g)$ for all
finite simple $H$. 
\end{proof}

Similarly, function values can be interpolated using functions from
$\scrs_t$.

\begin{theorem}[Lagrange Interpolation]\label{Plagrange}
Let $f_1, \dots, f_k \in \W/\sim$ be distinct graphons and $a_1, \dots ,
a_k \in \R$ be arbitrary real numbers. Then there exist finite simple
graphs $H_i$ and scalars $c_i$ such that $\sum_i c_i t(H_i,f_j) = a_j$
for all $1 \leq j \leq k$.  In other words, there exists an element $F
\in \scrs_t$ such that $F(f_j) = a_j$.
\end{theorem}

\begin{proof}
It suffices to show the result for $a_1 = 1$ and $a_j = 0$ for all $2 \le
j \le k$. Since the functions $t(H,-)$ separate points in $\W / \sim$,
for each $j \geq 2$ there exists a graph $H_j$ such that $t(H_j,f_j) \neq
t(H_j,f_1)$.  In particular there exist $b_j$ and $c_j$ such that $b_j
t(H_j,f_j) + c_j= 1$ and $b_j t(H_j,f_1) + c_j = 0$.  Recalling that the
linear span of homomorphism densities form an algebra, the function
$\prod_{j=2}^k (b_j t(H_j,f_j) + c_j)$ works.
\end{proof}\smallskip

\noindent {\bf Homomorphism densities as monomials.}

The above results suggest that $\scrs_t$ may play an important role in
the functional analysis of $(\W/\sim)$, for several reasons. For
instance, in addition to spanning an algebra of functions, the
homomorphism densities $t(H,f)$ naturally have a notion of degree. To
elaborate, let $\G_n$ be the set of isomorphism classes of unlabelled
simple graphs with $n$ edges and no isolated vertices. Now define
\begin{equation}
\G_{\leq n} := \bigcup_{j \leq n} \G_j, \qquad \G := \bigcup_{j \in \N}
\G_j.
\end{equation}

Equation \eqref{Eproduct} clearly shows that if $H_1 \in \G_n$ and $H_2
\in \G_{m}$ then $t(H_1,f) \times t(H_2,f) = t(H_3,f)$ for $H_3 \in \G_{n
+ m}$.  Therefore, the number of edges in the graph $H$ naturally serves
as a degree (i.e., a $\Z_+$-grading) for the function $t(H,f)$. Combined
with Theorems \ref{Stone} and \ref{Plagrange}, this suggests that
homomorphism densities may play the role of monomials in the algebra
$\scrs_t$. To carry this analogy further, recall that polynomials of
degree at most $N$ could be defined as solutions to the differential
equation
$\displaystyle \frac{d^{N+1}}{dx^{N+1}} F \equiv 0$.
A question of interest in the graphon setting would thus be to
ask which functions satisfy the system of differential equations
\begin{equation}\label{Ediff}
d^{N+1} F(f;g_1,\dots,g_{N+1}) \equiv 0, \qquad \forall f \in \W, \ g_i
\in \Adm(f).
\end{equation}

\noindent It is not hard to check that all homomorphism densities $F(f)
:= t(H,f)$ are solutions of \eqref{Ediff}, for $H \in \G_{\leq N}$.
Therefore Theorem \ref{thm:main} first of all gives further weight to our
notion of degree. Any function $\phi: \G \to \N$ satisfying $\phi (H_1
\coprod H_2) = \phi(H_1) + \phi(H_2)$ might be a candidate.   However, we
also show later that taking $N+1$ derivatives annihilates $t(H,-)$ for
$H$ with at most $N$ edges (see Proposition \ref{prop:highergateaux}).

What is more surprising is the fact that the homomorphism densities
$t(H,f)$ for $H \in \G_{\leq N}$ span all solutions of \eqref{Ediff}.
Therefore homomorphism densities (in fact for all multigraphs, not just
for all simple graphs) will play a fundamental role in the differential
calculus of class functions on $\W$.
We will make this analogy more precise at the end of Section
\ref{Sseries}.\bigskip

\noindent {\bf Continuity of multigraph homomorphism densities.}

Our main result, Theorem \ref{thm:main}, refers to homomorphism densities
$t(H,f)$ for $H$ a general multigraph without loops. For this paper, an
(undirected) multigraph $G$ is given by the data of the set of vertices
$V(G)$, set of undirected edges $E(G)$, and a map sending an edge $e$ to
its endpoints $\{e_s, e_t\} \subset V(G)$. We allow multiple edges to
have the same endpoints $\{e_s, e_t\}$, and the graph is undirected so
$e_s$ and $e_t$ could just as easily be interchanged.

Recall \cite[Section 5.2.1]{Lovasz:2013} that a {\it (node-and-edge)
homomorphism} of multigraphs $f: H \to G$ is defined by the data of a map
of vertices $V_f : V(H) \to V(G)$ and a map of edges $E_f : E(H) \to
E(G)$. The maps $E_f$ and $V_f$ must be compatible in the sense that
\begin{equation}\label{Enodeandedge}
\{E_f(e)_s,E_f(e)_t\} = \{V_f(e_s),V_f(e_t)\} \ \forall e \in E(H).
\end{equation}

\noindent Note that when $G$ is a multigraph, $E_f$ is not completely
determined by $V_f$. However, if $G$ is simple, we will identify $e \in
E(G)$ with its endpoints $\{ e_s,e_t \}$.
Moreover, we say that $f$ is (respectively) {\it injective}, {\it
surjective}, or {\it bijective}, when both $V_f$ and $E_f$ have the same
property.

Now let $\HH_n$ denote the isomorphism classes of graphs with $n$ edges,
no isolated vertices, and no self loops but possible multi-edges. Also
let $\HH_{\leq n} = \bigcup_{j \le n} \HH_j$. Clearly $\G_n \subset
\HH_n$ for all $n$. Then we have already defined
in Equation \eqref{Ehom} the homomorphism density $t(H,f)$
for an arbitrary multigraph $H \in \HH_n$, where $E(H)$ now denotes the
multiset of edges in $H$ and is independent of the choice of
representative $H$.
We can extend the definition of $t(H,-)$ to all $f \in \WW$.
When it is important to be more explicit about
naming the vertices and edges of $H$ we will write
\begin{equation*}
t(H,f) = \int_{[0,1]^{V(H)}} \prod_{e \in E(H)} f(x_{e_s},x_{e_t})
\prod_{i \in V(H)} dx_i, \qquad \forall H \in \HH, \ f \in \WW.
\end{equation*}

There is no consensus on the definition of a graph morphism between
multigraphs. One advantage of using the node-and-edge notion of
homomorphism is that if we define for multigraphs $H,G$ the combinatorial
quantity $t(H,G) := \hom(H,G)/|V(G)|^{|V(H)|}$, then
$t(H,G) =t(H,f^G)$
where $f^G$ is defined as for simple graphs, but weighted according to
the multiplicity of the edge.  So these multigraph homomorphism densities
are class functions and behave similarly to simple graph homomorphism
densities. We now show that multigraph homomorphism densities are no
longer continuous in the cut-norm topology unless they lie in $\scrs_t$.
The following proposition collects together the continuity properties of
homomorphism densities that are needed for the proof of Theorem
\ref{thm:main}.  The proof exploits the fact that $\{ 0, 1 \}$-valued
graphons are dense in $\W$.

\begin{prop}\label{prop:cont_func_char}
Fix $n \geq 0$ and consider $F : \WW \to \R$ of the form
$\displaystyle F(f) := \sum_{H \in \HH_{\leq n}} a_H t(H,f)$
for some constants $a_H$.
\begin{enumerate}[(i)]
\item Then $F : \WW \to \R$ is continuous in the $L^1$ topology.

\item If moreover $F$ is continuous in the cut-norm topology, then
\[ F(f) = \sum_{H \in \HH_{\leq n}} a_H t(H^{simp},f), \]

\noindent where $H^{simp}$ is the simple graph obtained from $H$ by
replacing each set of repeated edges between a pair of vertices by one
edge.
\end{enumerate}
\end{prop}

\begin{proof}\hfill
\begin{enumerate}[(i)]
\item It suffices to consider a single multigraph homomorphism density
$t(H,-)$.  Consider the multilinear functional $\Lambda : \W^n \to \R$
defined by
\begin{equation*}
\Lambda((f_e)_{e \in E(H)}) := \int_{[0,1]^{V(H)}} \prod_{e \in E(H)}
f_e(x_{e_s}, x_{e_t}) \prod_{i \in V(H)} dx_i.
\end{equation*}

\noindent It is not difficult to see, by replacing $f_e$ by $g_e$ one
term at time, that for $f_e, g_e \in \W$,
\begin{equation*}
\left|\Lambda((f_e)_{e \in E(H)}) - \Lambda((g_e)_{e \in E(H)}) \right|
\le \sum_{e \in E(H)} \|f_e - g_e\|_1
\end{equation*}

\noindent since $\|f_e\|_\infty, \|g_e\|_\infty \le 1$ for $f_e, g_e \in
\W$. It follows that $|t(H,f) - t(H,g)| \le |E(H)| \cdot \|f_e - g_e\|_1$
and so is continuous with respect to the $L^1$ topology.

\item  We can rewrite the function $t(H,f)$ as
\begin{equation*}
t(H,f) = \int_{[0,1]^k} \prod_{(i,j) \in E(H^{simp})} \varphi_{ij}(f(x_i,
x_j)) \prod_{i=1}^k dx_i
\end{equation*}  

\noindent where $\varphi_{ij}(x) = x^{m_{ij}}$, with $m_{ij}$ the
multiplicity of the edge $(i,j)$ in $H$.  From this expression, it is
clear that $t(H,f) = t(H^{simp},f)$ for $\{0,1\}$-valued graphons $f$.
In particular, $F(f) = \sum_{H \in \HH_{\leq n}} a_H t(H^{simp},f)$ for
$\{0,1\}$-valued graphons $f$. The result now follows from the continuity
of both sides in the cut-norm and the density of such graphons in $\W$.
\end{enumerate}

\end{proof}

\begin{remark}
In fact we show below that the homomorphism densities over all
multigraphs are linearly independent as functions on $\W$ -- see
Corollary \ref{cor:whitney}. As a consequence, if $F$ is given as in
Proposition \ref{prop:cont_func_char}
and is continuous in the cut-norm, then $a_H = 0$ for all $H$ that are
not simple. The converse also holds, because homomorphism densities for
simple graphs are continuous in the cut-norm.
\end{remark}

\subsection{Differentiation on \W}\label{Sdiff}

In this section we develop a general theory of differentiating functions
on graphon space. There are two standard notions of derivatives in such a
setting: the G\^ateaux derivative and the Fr\'echet derivative.
We show in this section that taking the Fr\'echet derivative is a very
restrictive notion and is not appropriate for our analysis, in that most
homomorphism densities are not Fr\'echet differentiable. We proceed to
develop some technical machinery to refine the G\^ateaux theory on $\W$
to helps bypass the fact that $\W$ is not a vector space (so one cannot
take G\^ateaux derivatives at all points along all directions). We also
make precise our notion of (sufficiently) smooth G\^ateaux differentiable
functions on $\W$. Finally, we illustrate our analytic methods by
providing a new proof of Sidorenko's Theorem for star graphs.\medskip

\noindent {\bf G\^ateaux derivatives, admissibility, and
smoothness}

The G\^ateaux derivative is usually defined in the context of a real
linear space $E$ and a map $F : E \to \mathbb{R}$.  In such settings the
{\it G\^ateaux derivative of $F$} at $f \in E$ in the direction of $g \in
E$ is defined to be the limit
$dF(f;g) := \lim_{\lambda \to 0} \frac{1}{\lambda}(F(f + \lambda g) -
F(f))$,
if such a limit exists. However, in this paper we have to
differentiate functions $F : \W \to \R$ where $\W$ is not a vector space.
In that case, if $f \in \W$ and $g \in E = \WW \supset \W$, then $f +
\lambda g$ need not always lie in $\W$, so we need to clarify what we
mean by the G\^ateaux derivative for functions on $\W$.  For instance, if
$f \equiv 0$ and $g$ is defined to be $-1$ and $1$ on disjoint
complementary subsets of $[0,1]^2$, then $f + \lambda g \notin \W$ for
any $\lambda \neq 0$. In this paper, we use the following notion to deal
with this issue.

\begin{definition}\label{Dadmissible}
Given a nonempty convex subset $U$ in a real linear space $E$ and $f \in
U$, define the {\it admissible directions at $f$} to be
\begin{equation}
\Adm{}_U(f) :=\{ g \in E : f + \epsilon g \in U \mbox{ for all
sufficiently small } 0 \leq \epsilon < 1 \}.
\end{equation}
\end{definition}

\begin{remark}
In this paper, unless otherwise specified, {\it admissibility} is
always assumed to be with respect to $E = \WW$ and $U = \W$. Thus we
always write $\Adm(f)$ to mean $\Adm_{\W}(f)$.
Also note that $\Adm_U(f)$ always contains the origin and is itself a
convex subset and a cone in $E$, since $U$ is convex.
\end{remark}

We now explain how the notion of admissibility applies to (higher)
G\^ateaux derivatives.

\begin{definition}\label{Dgateaux}
Let $E$ be a real linear space, $U \subset E$ a nonempty convex subset,
and $F : U \to \R$.
We say that the {\it G\^ateaux derivative} exists at $f \in U$ in the
direction $g \in \Adm_U(f)$ if the limit
$dF(f;g) := \lim_{\lambda \to 0,\ f + \lambda g \in U}
\frac{1}{\lambda}(F(f + \lambda g) - F(g))$
exists. Note that this limit is one-sided if $-g \notin \Adm_U(f)$.

Similarly, we say that $F$ is {\it $n$-times G\^ateaux differentiable}
at $f \in Y$ in the directions $g_1, \dots, g_n \in \Adm_U(f)$, if the
{\it higher mixed G\^ateaux derivatives} $d^{n-1}F(f + \lambda g_n; g_1,
\dots, g_{n-1})$ exist for all $\lambda$ small enough such that $f +
\lambda g_n \in U$, and the limit
\begin{equation*}
d^n F(f;g_1, \dots, g_n) := \lim_{\lambda \to 0,\ f + \lambda g_n \in U}
\frac{d^{n-1}F(f + \lambda g_n;g_1, \dots, g_{n-1}) - d^{n-1}F(f;g_1,
\dots, g_{n-1})}{\lambda}
\end{equation*}

\noindent exists.
\end{definition}

\begin{remark}
In the formula for defining higher G\^ateaux derivatives, one might
suspect that admissibility issues arise -- for instance, that $g_1 \in
\Adm_U(f + \lambda g_n)$ needs to hold for all small $\lambda$. However,
these issues are immediately bypassed since we assume $U$ to be convex.
Indeed, this is easily verified by induction on $n$, using the fact that
if $g_i \in \Adm_U(f)$ for all $i$, then $f + \sum_{i=1}^n \lambda_i g_i
\in U$ for all sufficiently small $0 \leq \lambda_i < 1$.
\end{remark}

In the remainder of this paper, we will need the notion of a continuously
differentiable function. This is made precise in the following
definition.

\begin{definition}\label{DCnatf}
Fix a real linear space $E$, a convex subset $U \subset E$, and a
function $F: U \to \mathbb{R}$.
Given $f \in U$, $m \in \N$, and $g_1, \dots, g_m \in \Adm_U(f)$, define
an auxiliary function $\mathcal{F}_{f, \bf g}(\lambda_1, \dots,
\lambda_m) := F(f + \lambda_1 g_1 + \dots + \lambda_m g_m)$. Note that
$\mathcal{F}_{f,  \bf g}$ is defined on a convex subset of $\mathbb{R}^m$
containing zero.  Now given an integer $n \geq 0$, we say that $F$ is
{\it $C^n$ at $f$} if $\mathcal{F}_{f, \bf g}$ has a $C^n$ extension to
all of $\mathbb{R}^m$ for all $g_1, \dots, g_m \in \Adm_U(f)$ and $m \in
\N$. We say that $F$ is {\it smooth at $f$} if it is $C^n$ at $f$ for
every $n \geq 0$. We say that $F$ is $C^n$ (or smooth) if it is $C^n$ (or
smooth) everywhere.
\end{definition}

\begin{remark}
The question of the possibility of extending a function $f$ and its
candidate derivatives $(f_\alpha)_{|\alpha| \le n}$ on a closed domain $D
\subset \mathbb{R}^m$ to a $C^n$ function and its derivatives on all of
$\mathbb{R}^m$ is addressed by Whitney-type extension theorems  (see for
example the exposition in \cite{BB:2012}).  In this paper, we will be
able to find extensions explicitly for the functions of interest -- see
Lemma \ref{lem:multilinear_deriv}.
\end{remark}


In order to show that $d^n F(f; g_1, \dots, g_n)$ is multilinear in the
$g_i$, we make the following definition.

\begin{definition}\label{Dlinear}
Suppose $V \subset E$ is a subset of a real vector space $E$. A function
$\Lambda : V^n \to \R$ is said to be {\it multilinear} if $\Lambda$
extends (uniquely) to a multilinear functional $: ({\rm span}_\R V)^n \to
\R$.
\end{definition}

We now write down a precise statement about the multilinearity of the
(higher) G\^ateaux derivatives. This lemma is crucial for the rest of
the paper.

\begin{lemma}\label{lem:std_gateaux}
Suppose $U \subset E$ is convex, and the function $F : U \to \R$ is $C^n$
at $f \in U$.
\begin{enumerate}[(i)]
\item For any permutation $\tau \in S_n$, we have $d^nF(f;g_1, \dots,
g_n) = d^nF(f;g_{\tau(1)}, \dots, g_{\tau(n)})$ for all $g_i \in
\Adm_U(f)$ for $1 \leq i \leq n$.
\item The G\^ateaux derivatives $d^n F(f;g_1, \dots, g_n)$ are
multilinear in the $g_i \in \Adm_U(f)$.
\end{enumerate}
\end{lemma}

\begin{proof}
The functions $\mathcal{F}_{f, \bf h}$ for $h_1, \dots, h_m \in
\Adm_U(f)$ can be extended to $C^n$ functions on $\mathbb{R}^m$.
Therefore part (i) follows from choosing $m = n$ and $h_i = g_i$ and the
equality of mixed partials.  

To show part (ii), we must show that $d^n F(f;g_1, \dots, g_n)$ extends
to a unique multilinear functional on $(g_1, \dots, g_n) \in ({\rm
span}_\R \Adm_U(f))^n$.  Since $({\rm span}_\R \Adm_U(f))^n$ is the union
of subspaces $({\rm span}_\R (h_1, \dots, h_m))^n$ for all $m \in \N$ and
$h_i \in \Adm_U(f)$, it suffices to show that $d^n F(f;g_1, \dots, g_n)$
extends multilinearly to each $({\rm span}_\R (h_1, \dots, h_m))^n$.  Let
${\bf h} := (h_1, \dots, h_m)$. By the definition of $C^n$ functions,
$\mathcal{F}_{f, \bf h}$ can be extended to a $C^n$ function on
$\mathbb{R}^m$.  Its derivatives can now be used to extend $d^n F(f;g_1,
\dots, g_n)$ to $({\rm span}_\R (h_1, \dots, h_m))^n$. It is not hard to
see that the different extensions for $\bf h$ are consistent so the
result follows.
\end{proof}


%
%
%

In later sections, we will almost always assume that our functions are
G\^ateaux smooth or $C^n$. Indeed, we need these properties (including
Lemma \ref{lem:std_gateaux}) in order to prove Theorem \ref{thm:main}
and other main results.  \bigskip


\begin{remark}
A stronger notion of differentiability of functions $F : E \to \R$ which
is often used on normed linear spaces $E$ is the Fr\'echet derivative. It
is natural to ask if such a notion can be used to study functions on
$\WW$ equipped with the seminorm $\|\cdot\|_{\cut}$.  However, even
homomorphism densities are generally not Fr\'echet differentiable, as we
now explain.

Let $A_n$ denote the graph with vertex set $\{1,2,3, \dots, n\}$ and
edges $(i,i+1)$ for $1 \le i \le n-1$.  If $H$ is a disjoint union of
copies $A_1$ and $A_2$ then it is easy to see that $t(H,f)$ is Fr\'echet
differentiable on all of $\WW$.  However, for general $H$, the formula
for $t(H,f)$ may not define a continuous function on $\WW$.  Even if we
restrict to $\W$ we can compute that for $f \ge c > 0$, the G\^ateaux
derivatives of $t(H,f)$ at $f$ do not form a good enough linear
approximation to $t(H,f)$. To see this, define $g_n \in \W$ via:
$g_n(x_1,x_2) := {\bf 1}(\min(x_1,x_2) < n^{-1})$. One now checks that if
$f \in \WW$ is bounded below by $c>0$, and $H$ is not a disjoint union of
copies of $A_1$ and $A_2$, then
\[ \lim_{n \to \infty} \frac{|t(H,f + g_n) - t(H,f) - d(t(H,f);
g_n)|}{\boxnorm{g_n}} \geq \frac{1}{2} c^{|E(H)|-2} > 0. \]

\noindent We therefore work with G\^ateaux derivatives in the rest of the
paper.
\end{remark}

G\^ateaux derivatives and other variational techniques have been used to
investigate problems in combinatorics and graph theory in the literature;
see for instance \cite{CD:2011}, \cite[Chapter 16.2]{Lovasz:2013}.
We illustrate how the G\^ateaux derivative can be used to solve
optimization problems via a simple case of Sidorenko's conjecture --
namely, for star graphs. This case was solved in Sidorenko's original
paper \cite{Sidorenko:1993}.

\begin{theorem}\label{thm:sidorenkostars} 
Let $S_k$ be the star graph with $k+1$ vertices $\{0,1,2, \dots, k\}$
and $k$ edges from $0$ to all $i>0$.
\begin{enumerate}[(i)]
\item If $f \in \W$ has edge density $t(K_2, f) = c$ then $t(S_k,f) \ge
t(S_k,c)$. 
\item If $f \in \W$ has edge density $t(K_2, f) = c$ then $t(S_k,f) =
t(S_k,c)$ if and only if $\int_x f(x,y) = c$ for almost every $y$. 
\end{enumerate}  
\end{theorem}

\begin{proof}

\begin{enumerate}[(i)]
\item  Let $\mathcal{W}^{(0)} \subset \WW$ denote the linear subspace of
all $f \in \WW$ with edge density $0$. 

\noindent Either by direct computation or by Proposition
\ref{prop:highergateaux} we can compute the higher derivatives of
$t(S_k,f)$ and see that it is G\^ateaux smooth on $\W$.  Let $g := f - c$
and note that $g \in \WW^{(0)}$.  The first derivative of $t(S_k,-)$ at
$c$ is
\begin{equation}\label{Estarlocalminimum}
dt(S_k,c;g) = \sum_{(i_1,j_1)\in E(S_k)} \int_{[0,1]^{k+1}}
g(x_{i_1},x_{j_1}) \left( \prod_{(i,j) \neq (i_1,j_1)} c \right)
\prod_{i=0}^k dx_i = 0,
\end{equation}
since $g \in \WW^{(0)}$.  

\noindent The second derivative at any $f \in \W$ and $g \in \Adm(f)$ is
given by
\begin{equation}\label{Estarsecondderivative}
d^2t(S_k,f;g,g) = 2\sum_{1 \le i < j \le k} \int_{[0,1]}
\left(\int_{[0,1]} g(x_0,z) dz\right)^2 \left(\int_{[0,1]} f(x_0,y)
dy\right)^{k-2} dx_0.
\end{equation}

\noindent Since $f \ge 0$ and $(\int_{[0,1]} g(x_0,z) dz)^2 \ge 0$, we
conclude that 
$d^2t(S_k,f;g,g) \ge 0$.
Now consider the set
\[ I(f,c) := \{ \lambda c + (1-\lambda)f : \lambda \in [0,1] \} \subset c
+ \WW^{(0)}. \]

\noindent By Equation \eqref{Estarlocalminimum}, we conclude that the
constant graphon $c$ is a local minimum of $t(S_k,f)$ on $I(f,c)$.
The star density $t(S_k,-)$ is convex on $I(f,c)$
since from above, $d^2t(S_k, \lambda f + (1-\lambda)c ;f-c,f-c) \geq 0$
for all $\lambda \in [0,1]$. Therefore $t(S_k,f) \geq t(S_k,c)$.

\item Assume now that $\int_x f(x,y) dx = c$ for almost every $y$. 
We compute
\[ t(S_k,f) = \int_{[0,1]} \prod_{1 \le i \le k} \left(\int_{[0,1]}
f(x_0,x_i) dx_i\right) dx_0 = \int_{[0,1]} c^k dx_0 = t(S_k,c). \]

\noindent Assume now that $\int_x f(x,y)$ is not equal to a constant for
almost every $y$. Let $g := f - c$. Using Equation
\eqref{Estarsecondderivative} we compute that the second derivative is
given by
\[ d^2t(S_k,c;g,g) = 2 c^{k-2} \sum_{1 \le i < j \le k} \int_{[0,1]}
\left(\int_{[0,1]} g(x_0,z) dz\right)^2 dx_0. \]

\noindent This last integral is positive because $\int_x f(x,y)$ is not
equal to a constant for almost every $y$.  It follows that $t(S_k, -)$ is
strictly convex on $I(f,c)$ so $t(S_k,f) > t(S_k,c)$.  
\end{enumerate}
\end{proof}

\section{Derivatives of $C^N$ class functions}\label{S4}

The main goal of this section is to prove Theorem \ref{thm:main}. As the
proof of Theorem \ref{thm:main} is long and technical, we begin with an
overview of the ingredients that will be used to prove it. The main
ingredients have been separated out into subsections for ease of
presentation. We begin by investigating the derivatives of smooth class
functions as developed in Section \ref{Sdiff}. To explain that
connection, consider the differential equation of Theorem \ref{thm:main}
one direction at a time -- i.e., $d^{N+1}F(f;g,g,\dots,g) = 0$ for all $f
\in \W$ and some fixed $g \in \W$. This differential equation has
solutions $F$ that satisfy:
\begin{equation*}
F(g) = F(0) + dF(0;g) + \frac{d^2F(0;g,g)}{2!} + \dots +
\frac{d^nF(0;g,\dots, g)}{n!}.
\end{equation*}

\noindent In other words, $F$ is determined by the initial data of its
derivatives at $0$.  From this perspective the differential equation of
the main theorem could have an uncountable dimensional space of solutions
corresponding to the different possible derivatives at $0$.  
In fact, this is not the case because any solution $F$ to Theorem
\ref{thm:main} is a smooth class function, and the derivatives of any
smooth class function satisfy the following two important symmetry
properties. 

\begin{definition}
Fix $n \in \N$. A functional $\Lambda: \W^n \to \R$ is said to be
\begin{itemize}
\item {\it symmetric} if for all permutations $\tau \in S_n$,
$\Lambda(g_1, \dots, g_n) = \Lambda(g_{\tau(1)}, \dots, g_{\tau(n)})$;
\item \textit{$S_{[0,1]}$-invariant} if $\Lambda(g_1^\sigma, \dots,
g_n^\sigma) = \Lambda(g_1, \dots, g_n)$ for all Lebesgue measure
preserving bijections $\sigma: [0,1] \to [0,1]$.
\end{itemize}
\end{definition}

\begin{prop}\label{prop:Cn_consistent}
If $F : \W \to \R$ is a $C^n$ class function for some integer $n>0$, then
$d^n F(0; g_1, \dots, g_n)$ is a symmetric $S_{[0,1]}$-invariant
multilinear functional.
\end{prop}

\begin{proof}
By part (i) of Lemma \ref{lem:std_gateaux}, we know for $C^n$ functions
$F: \W \to \R$ that mixed $n$th partial G\^ateaux derivatives are equal.
Therefore,  $d^n F(0; g_1, \dots, g_n)$ is symmetric.  By part (ii) of
Lemma \ref{lem:std_gateaux}, we also get that the derivative is
multilinear.

Next, note that if $F : \W \to \R$ is an $n$-times G\^ateaux
differentiable class function, $\sigma \in S_{[0,1]}$, and $g_i \in
\Adm(f)$ for $f \in \W$, then $g_i^\sigma \in \Adm(f^\sigma)$ and
\begin{equation}
d^nF(f^\sigma; g_1^\sigma, \dots, g_n^\sigma) = d^nF(f;g_1, \dots, g_n).
\end{equation}

\noindent Applying this equation to $f = 0$, we obtain that $d^n
F(0;g_1,\dots, g_n)$ is $S_{[0,1]}$-invariant.
\end{proof}

Let $\lie{X}_n$ denote the vector space of symmetric
$S_{[0,1]}$-invariant multilinear functionals $\Lambda : \W^n \to \R$ for
$n \geq 1$. Note by Proposition \ref{prop:Cn_consistent} that the
derivatives at zero of the solutions of the differential equation in
Theorem \ref{thm:main} all lie in $\lie{X}_n$. In Section
\ref{section:cc}, we study the space $\lie{X}_n$ via its image under
linear maps $C_{n,p} : \lie{X}_n \to X_{n,p}$ for $p \geq 2$. Here,
$X_{n,p} = \R^{\HH_n^{(p)}}$ with $\HH_n^{(p)}$ the set of isomorphism
classes of multigraphs with $n$ edges and $p$ vertices. Notice that
$\HH_n^{(p)}$ is not necessarily a subset of $\HH_n$ because it allows
for isolated vertices.
We show for each $\Lambda \in \lie{X}_n$ that the value of $\Lambda$
restricted to $n$-tuples of edge-weighted graphs is determined by
$C_n(\Lambda) := (C_{n,p}(\Lambda))_{p \geq 2}$. In addition, we show for
$p|q$ that there are linear relations $\pi_{n,q \to p}$ called the
\textit{consistency constraints} mapping $X_{n,q}$ to $X_{n,p}$, which
send $C_{n,q}(\Lambda)$ to $C_{n,p}(\Lambda)$. The upshot is that the
image $C_n(\lie{X}_n)$ has dimension at most $|\HH_n|$.  

In Section \ref{section:mainproof}, we first note that the $n$th
G\^ateaux derivatives at $0$ of $\{t(H,-)\}_{H \in \HH_n}$ are in
$\lie{X}_n$.  Next we show that $C_n(d^n t(H,-)(0;-))$ are linearly
independent for $H \in \HH_n$. Therefore by counting dimensions, the
image $C_n(\lie{X}_n)$ is spanned by the $n$th G\^ateaux derivatives of
homomorphism densities for $H \in \HH_n$. Finally, we collect the
different solutions for $n \leq N$ and use the continuity assumptions as
in Proposition \ref{prop:cont_func_char} to conclude the proof.

\subsection{Symmetric $S_{[0,1]}$-invariant multilinear
functionals}\label{section:cc}

The main goal of this section is to investigate symmetric
$S_{[0,1]}$-invariant multilinear functionals $\Lambda: \W^n \to \R$,
where $n \in \N$.

\begin{definition}
For integers $1 \le a < b \le p$, define 
\begin{equation}
e^p_{(a,b)} := {\bf 1}_{\left( \frac{a-1}{p}, \frac{a}{p} \right] \times
\left( \frac{b-1}{p}, \frac{b}{p} \right]} + {\bf 1}_{\left(
\frac{b-1}{p}, \frac{b}{p} \right] \times \left( \frac{a-1}{p},
\frac{a}{p} \right]}.
\end{equation}

\noindent Now define $E_p := \{e^p_{(a,b)} : 1 \le a < b \le p\}$, and
\begin{equation}
\WW_p := \W \cap {\rm span}_\R E_p, \qquad \WW_{\bf p} :=
\bigcup_{p=1}^\infty \WW_p.
\end{equation}
\end{definition}

We classify the different symmetric $S_{[0,1]}$-invariant multilinear
functionals restricted to $\WW_{\bf p}^n$ by defining constants that
determine them. Let $\Lambda: \W^n \to \R$ be a multilinear functional.
By multilinearity, the restriction of $\Lambda$ to $\WW_{\bf p}$ is
determined by the infinite set of constants $(\Lambda(e))_{e \in
E^n_{p},\ p \geq 2}$. Surprisingly, once we assume that $\Lambda$ is
symmetric and $S_{[0,1]}$-invariant, $\Lambda$ is determined by only a
finite number of these constants.  To prove this, we investigate the
relations between the $\Lambda(e)$ for $e \in E^n_p$.

We begin by defining and explaining some basic notation that is used in
the proof of Theorem \ref{thm:main}. First note that there is a group
action of the symmetric group $S_p$ on $E_p$ and therefore on $E_p^n$,
defined by $\sigma (e^p_{(a,b)}) := e^p_{(\sigma(a),\sigma(b))}$.  There
is also an $S_n$ action on $E_p^n$ defined by permuting the coordinates
of $(e^p_{(a_l,b_l)})_{l=1}^n$.  The $S_p$ and $S_n$ actions commute so
together they define an $S_p \times S_n$ action on $E_p^n$.

There is a natural map way to associate a multigraph to any tuple $x =
(e^p_{(a_l,b_l)})_{l=1}^n \in E_p^n$, with vertex set $\{1, \dots, p\}$
and edges $\{x(l): (x(l)_s,x(l)_t) = (a_l,b_l), 1 \le l \leq n \}$.
Denote this multigraph by $\Gamma_{n,p}((e^p_{(a_l,b_l)})_{l=1}^n)$.
Given any multigraph $G$, let $[G]$ denote the equivalence class of
multigraphs isomorphic to $G$.

For simplicity, we will often drop either of the subscripts in the
notation for $\Gamma$ when there is no chance of confusion.  If $h \in
\HH_n^{(p)}$, let $\novert{h} \in \HH_n$ be the graph obtained by
removing the isolated vertices of $h$. Similarly, if $H$ is any
multigraph, then denote by $\novert{H}$ the multigraph obtained by
removing the isolated vertices of $H$. The following proposition
summarizes the basic properties of the map $\Gamma_{n,p}$.

\begin{prop}[Properties of $\Gamma_{n,p}$]\label{prop:gamma}\hfill
\begin{enumerate}[(i)]
\item The map sending $x \in E^n_p$ to $[\Gamma_{n,p}(x)] \in
\HH_n^{(p)}$ is surjective.
\item The fibers of the map sending $x \in E^n_p$ to $[\Gamma_{n,p}(e)]
\in \HH_n^{(p)}$ are precisely the $S_p\times S_n$-orbits.
\item The map $h \to \novert{h}$ sending an element of $\HH_n^{(p)}$ to
$\HH_n$ by removing the isolated vertices is injective.  In addition, it
is bijective if and only if $p \geq 2n$.  
\end{enumerate}
\end{prop}

\begin{proof}\hfill
\begin{enumerate}[(i)]
\item Let $G$ denote a representative of a class in $\HH_n^{(p)}$, and
fix bijections
\[ \phi: V(G) \to \{1, \dots, p\}, \qquad \psi : E(G) \to \{ 1, \dots, n
\}. \]

\noindent Now define $x = (e^p_{(a_i, b_i)})_{i=1}^n \in E_p^n$ via: $a_i
:= \phi(\psi^{-1}(i)_s), b_i := \phi(\psi^{-1}(i)_t)$. Clearly
$\Gamma_{n,p}(x)$ and $G$ are isomorphic so $\Gamma_{n,p}$ is surjective
onto $\HH_n^{(p)}$.

\item Let $x \in E_p^n$, $\sigma \in S_p$, and $\tau \in S_n$.  Then
$\Gamma_{n,p}(x)$ and $\Gamma_{n,p}(\tau(x))$ are the same multigraph and
$\Gamma_{n,p}(x)$ and $\Gamma_{n,p}(\sigma(x))$ are clearly isomorphic
multigraphs.

Conversely, if $\Gamma_{n,p}(x_1)$ is isomorphic to $\Gamma_{n,p}(x_2)$
then there exist two maps $V_f:V(\Gamma_{n,p}(x_1)) \to
V(\Gamma_{n,p}(x_2))$ and $E_f: E(\Gamma_{n,p}(x_1)) \to
E(\Gamma_{n,p}(x_2))$ that form an isomorphism of multigraphs $f$.  Note
that $V_f \in S_p$ because it is a bijection $\{1, \dots, p\} \to \{1,
\dots, p\}$.  If $x_1 = (e^p_{(a_l,b_l)})_{l=1}^n$ and $x_2 =
(e^p_{(c_l,d_l)})_{l=1}^n$ then the bijection $E_f$ defines a bijective
map $\tau: \{1, \dots, n\} \to \{1, \dots, n\}$ by sending $i \to j$ if
the $i$th edge $(a_i,b_i)$ maps to the $j$th edge $(c_j,d_j)$.  Then,
$(V_f,\tau) \in S_p \times S_n$ and $(V_f,\tau)(x_1) = x_2$.  

\item The map $h \to \novert{h}$ is clearly injective.  The graph with
the most number of vertices in $\HH_n$ is the one with $n$ disjoint edges
-- i.e., $A_2^{\coprod n}$.  This has $2n$ vertices and so when $p \geq
2n$, $h \to \novert{h}$ surjects onto $\HH_n$, while $A_2^{\coprod n}$
does not lie in the image when $p < 2n$.
\end{enumerate}
\end{proof}

For fixed $p \geq 2$, we now show how Proposition \ref{prop:gamma} allows
us to define constants associated to $\Lambda \in \lie{X}_n$ indexed by
$h \in \HH_n^{(p)}$, that carry all of the information of
$(\Lambda(x))_{x \in E^n_p}$. In particular, since $\Lambda$ is symmetric
and $S_{[0,1]}$-invariant, it is invariant under the $S_n$ and $S_p$
actions respectively.

\begin{definition}\label{remark:part_class}
Let $\Lambda: \W^n \to \R$ be a symmetric $S_{[0,1]}$-invariant
multilinear functional, i.e., $\Lambda \in \lie{X}_n$. Then for any $p
\geq 2$ and $h \in \HH_n^{(p)}$, pick by Proposition \ref{prop:gamma}(i)
an $x \in E_p^n$ such that $h = [\Gamma_{n,p}(x)]$, and define
\begin{equation*}
C_{n,p}(\Lambda)(h) := \Lambda(x).
\end{equation*}
The value of $C_{n,p}(\Lambda)(h)$ does not depend on the choice of $x$
by Proposition \ref{prop:gamma}(ii).

\noindent Also define the map $C_n : \lie{X}_n \to \prod_{p \geq 2}
X_{n,p}$ where $X_{n,p} = \R^{\HH_n^{(p)}}$ by
\begin{equation*}
C_n(\Lambda) := (C_{n,p}(\Lambda))_{p \geq 2}, \qquad \mbox{where} \qquad
C_{n,p}(\Lambda) := (C_{n,p}(\Lambda)(h))_{h \in \HH_n^{(p)}}.
\end{equation*}
\end{definition}

The following theorem reveals the relations between the vectors
$C_{n,p}(\Lambda)$.  We shall see that the vectors necessarily satisfy
certain compatibility conditions, for a fixed $\Lambda$ and varying $p
\in \N$. More surprisingly, we now show that for each $n, k \in \N$,
there exists a {\it single} matrix that determines the compatibility
constraints, across {\it all} $\Lambda \in \lie{X}_n$ and all $p \geq 2$.

\begin{theorem}[Consistency Relations]\label{thm:consistency}
Fix $n,k \in \N$. There exists a fixed matrix $\pi_{n,k}  \in \Z_{\geq
0}^{\HH_n \times \HH_n}$ such that for any $p \geq 2$, $\Lambda \in
\lie{X}_n$, and $g \in H_n^{(p)}$, we have:
\begin{equation}\label{Econsistency2}
C_{n,p}(\Lambda)(g) = \sum_{h \in \HH_n^{(kp)}}
\pi_{n,k}(\novert{g},\novert{h}) C_{n,kp}(\Lambda)(h), \qquad \forall
g \in \HH_n^{(p)}.
\end{equation}
In addition, $\pi_{n,k}(\novert{g}, \novert{h})$ is nonzero only if
there exists $H \in \novert{h}$ that surjects onto $G \in \novert{g}$
as a multigraph, and $\pi_{n,k}(\novert{g}, \novert{g}) > 0$.
\end{theorem}

\begin{proof}
Write the basis elements of $E_p$ in terms of $k^2$ basis elements in
$E_{kp}$ as follows:
\begin{equation}\label{eqn:splitting}
e^p_{(a,b)} = \sum_{i,j=1}^k e^{kp}_{(k(a-1) + i,k(b-1) + j)}.
\end{equation}

\noindent Now choose any $x = (e^p_{(a_l,b_l)})_{l=1}^n \in E_p^n$ such
that $g = [\Gamma_{n,p}(x)]$, and expand
\begin{equation}\label{eqn:consistency_index}
C_{n,p}(\Lambda)(g) = \Lambda((e^{p}_{(a_l,b_l)})_{1 \leq l \leq n}) =
\sum_{(i_l,j_l)_l \in \{1, \dots, k\}^{2n}} \Lambda((e^{kp}_{(k(a_l-1) +
i_l,k(b_l-1) + j_l)})_{1 \leq l \leq n})
\end{equation}

\noindent by splitting up each basis element using equation
\eqref{eqn:splitting} and multilinearity. For every choice of
$(i_l,j_l)_l \in \{1, \dots, k\}^{2n}$, define a graph
\begin{equation}\label{Eindextograph}
H((i_l,j_l)_{l=1}^n) =  \Gamma_{n,kp}((e^{kp}_{(k(a_l-1) + i_l,k(b_l-1) +
j_l)})_{1 \leq l \leq n}).
\end{equation}

\noindent We can then rewrite Equation \eqref{eqn:consistency_index} as
\begin{equation*}
C_{n,p}(\Lambda)(g) = \Lambda((e^{p}_{(a_l,b_l)})_{1 \leq l \leq n}) =
\sum_{(i_l,j_l)_l \in \{1, \dots, k\}^{2n}}
C_{n,kp}(\Lambda)([H((i_l,j_l)_{l=1}^n)]).
\end{equation*}

Let the map $\alpha_{n,k,p,x}: \{1, \dots k \}^{2n} \to \HH_n$ be defined
by sending $(i_l,j_l)_{l=1}^n$ to $[\novert{H}((i_l,j_l)_{l=1}^n)]$ and
let $M(n,k,p,x,h)$ be the size of the fiber of $\alpha_{n,k,p,x}$ over $h
\in \HH_n$.  Then
\begin{equation*}
C_{n,p}(\Lambda)(g) = \sum_{h \in \HH_n^{(kp)}} C_{n,kp}(\Lambda)(h)
\cdot M(n,k,p,x,\novert{h}).
\end{equation*}

\noindent Using the $S_p \times S_n$ action on $\alpha_{n,k,p,x}$, one
verifies that $M(n,k,p,x,h) = M(n,k,p,(\sigma,\tau)(x),h)$ for all
$(\sigma,\tau) \in S_p \times S_n$. Hence by Proposition
\ref{prop:gamma}(ii), $M(n,k,p,[\Gamma_{n,p}(x)],h) :=
M(n,k,p,x,h)$ is well-defined, and
\begin{equation*}
C_{n,p}(\Lambda)(g) = \sum_{h \in \HH_n^{(kp)}} C_{n,kp}(\Lambda)(h)
\cdot M(n,k,p,g,\novert{h}).
\end{equation*}

We now claim that for any $g \in \HH_n^{(p)}$ and $h \in \HH_n$,
$M(n,k,p,g,h) = \pi_{n,k}(\novert{g},h)$ for some fixed matrix
$\pi_{n,k} \in \Z_{\geq 0}^{\HH_n \times \HH_n}$ independent of $p$ and
$\Lambda$.  Indeed, given integers $2 \leq p \leq p'$ and $g \in
\HH_n^{(p)}$, choose $x := (e^p_{(a_l,b_l)})_{l=1}^n$ with
$[\Gamma_{n,p}(x)] = g$. Now define $x' := (e^{p'}_{(a_l,b_l)})_{l=1}^n$
and $g' := [\Gamma_{n,p'}(x')]$; then $[\novert{g}] = [\novert{g'}]$.
Moreover, $\alpha_{n,k,p,x}((i_l,j_l)_{l=1}^n) =
\alpha_{n,k,p',x'}((i_l,j_l)_{l=1}^n)$ for all $((i_l,j_l)_{l=1}^n \in
\{1, \dots, k\}^{2n}$.  Therefore since $M(n,k,p,g,h)$ is the size of the
fiber of $\alpha_{n,k,p,x}$ and $M(n,k,p',g',h)$ is the size of the
fiber of $\alpha_{n,k,p',x'}$ over $h \in \HH_n$, $M(n,k,p,g,h) =
M(n,k,p,g',h) =:  \pi_{n,k}(\novert{g},h)$. 

It remains to show the last sentence of the result. Suppose $\pi_{n,k}(g,
h) > 0$ for $g,h \in \HH_n$. Pick arbitrary fixed $p \ge 2n$ and $x =
(e^p_{(a_l,b_l)})_{l=1}^n \in E_p^n$ such that
$[\novert{\Gamma}_{n,p}(x)] = g$ by Proposition \ref{prop:gamma}. Then
there exists $(i_l,j_l)_{l=1}^n \in \{1, \dots, k\}^{2n}$ such that
$h = [\novert{H}((i_l,j_l)_{l=1}^n)]$ (see Equation
\eqref{Eindextograph}) by the above analysis. There is an obvious
surjective map from $H((i_l,j_l)_{l=1}^n)$ to $\Gamma_{n,p}(x)$ given by
sending the vertex $a$ to the vertex $\lfloor(a-1)/k\rfloor + 1$ and
sending the $l$th edge of $H((i_l,j_l)_{l=1}^n)$ to the $l$th edge of
$\Gamma_{n,p}(x)$. Therefore, $\pi_{n,k}(g,h) > 0$ implies that there
exists a surjective map from a multigraph $H \in h$ to a multigraph $G
\in g$. In addition, picking $(i_l, j_l) = (1,1)$ for $1 \le l \le n$
shows that $\pi_{n,k}(g,g) > 0$.
\end{proof}

\begin{definition}
Given $n,k, 2 \leq p \in \N$, define the map $\pi_{n,kp \to p} : X_{n,kp}
\to X_{n,p}$ as follows: $\pi_{n,kp \to p}$ sends the vector $A_{n,kp} =
(A_{n,kp}(h))_{h \in \HH_n^{(kp)}}$ to the vector $A_{n,p} =
(A_{n,p}(g))_{g \in \HH_n^{(p)}}$, where
\[ A_{n,p}(g) = \sum_{h \in \HH_n^{(kp)}} \pi_{n,k}(\novert{g},
\novert{h}) A_{n,kp}(h). \]

We call the linear maps $\pi_{n,kp \to p}$ the \textit{consistency
constraints}.  We also say that any vector $A = (A_{n,p})_{p \geq 2} \in
\prod_{p \ge 2} X_{n,p}$ satisfying the constraints
$A_{n,p} = \pi_{n,kp \to p}(A_{n,kp})$ is \textit{linearly consistent}.
\end{definition}

Using the consistency constraints, we now prove that multilinear
functionals $\Lambda \in \lie{X}_n$ restricted to $\WW_{\bf p}$ are
determined by $|\HH_n|$ constants. 

\begin{theorem}\label{thm:upper}
Fix $n,k \in \N$. Then the following hold.
\begin{enumerate}[(i)]
\item The matrix $\pi_{n,k}$ is triangular with positive diagonal
entries when $\HH_n$ is partially ordered by the existence of a
surjective map of multigraphs.

\item The $\pi_{n,kp \to p}$ are  surjective maps that are
invertible for $p \geq 2n$, and compatible in the following sense: given
positive integers $n,k_1,k_2,2 \leq p$,
\begin{equation}\label{Ecompatible}
\pi_{n,k_1 k_2 p \to p} = \pi_{n, k_2 p \to p} \circ \pi_{n, k_1 k_2 p
\to k_2 p}.
\end{equation}

\item For each $n \in \N$, the subspace $LC_n \subset \prod_{p \geq 2}
X_{n,p}$ of linearly consistent vectors
\begin{equation*}
(A_{n,p}(h))_{h \in \HH^{(p)}_n,p \geq 2} \in \prod_{p \geq 2} X_{n,p}
\end{equation*}
has dimension $|\HH_n|$.
 
\item If $\Lambda \in \lie{X}_n$, the $|\HH_n|$ components of
$C_{n,p_0}(\Lambda) \in X_{n,p_0}$ for any $p_0 \ge 2n$ determine the
value of $\Lambda$ on $\sigma(\WW_{\bf p}^n)$ for all $\sigma \in
S_{[0,1]}$.
\end{enumerate}
\end{theorem}

\begin{proof}\hfill
\begin{enumerate}[(i)]
\item Note that the graphs in $\HH_n$ are partially ordered by the
existence of a surjective map of multigraphs.  By Theorem
\ref{thm:consistency}, $\pi_{n,k}([\novert{g}],[\novert{h}]) > 0$ for
any two multigraphs $g,h$ with $n$ edges only if there exists a
surjective map from $h$ to $g$. Thus $\pi_{n,k}$ is triangular when
$\HH_n$ is ordered with respect to the existence of a surjective map.  In
addition Theorem \ref{thm:consistency} states that
$\pi_{n,k}([\novert{g}],[\novert{g}]) > 0$, so the diagonal entries
of $\pi_{n,k}$ are positive.

\item It is easy to see that $\HH_n^{(p)} \hookrightarrow \HH_n^{(kp)}$.
Let the corresponding subspace of $\R^{\HH_n^{(kp)}} = X_{n,kp}$ be
called $Y_{n,p,k}$. We now claim that $\pi_{n,kp \to p} : Y_{n,p,k} \to
X_{n,p}$ is an isomorphism -- in particular, it is surjective. Indeed,
this is obvious since the restriction of $\pi_{n,kp \to p}$ to
$Y_{n,p,k}$ is given by a principal submatrix of $\pi_{n,k}$, which is
itself triangular with nonzero diagonal entries.
Now if $p \ge 2n$, the maps $\pi_{n,kp \to p}$ are invertible because
$Y_{n,p,k} = X_{n,kp}$.

Finally, to show that $\pi_{n,k_2p \to p} \circ \pi_{n,k_1k_2p \to k_2p}
= \pi_{n,k_1k_2p \to p}$, note that expanding basis elements in $E_p$ by
Equation \eqref{eqn:splitting} into basis elements elements in
$E_{k_1k_2p}$, via
\begin{equation*}
e^p_{(a,b)} = \sum_{i,j=1}^{k_1k_2} e^{k_1k_2p}_{(k_1k_2(a-1) +
i,k_1k_2(b-1) + j)}
\end{equation*}
is the same as expanding $e^p_{(a,b)}$ into basis elements in $E_{k_2p}$
and then splitting those basis elements into basis elements in
$E_{k_1k_2p}$. The proof follows by using counting arguments as in the
proof of Theorem \ref{thm:consistency}.

\item Fix $p_0 \geq 2n$. We show that the map $P : LC_n \to X_{n,p_0}$
sending $A = (A_{n,p})_{p \geq 2}$ to $A_{n,p_0}$ is a linear
isomorphism. Indeed, $P$ is injective because if $A_{n,p_0} = 0$, then by
(ii),
\[ A_{n,p} = \pi_{n, p_0 p \to p} \circ \pi_{n, p_0 p \to
p_0}^{-1}(A_{n,p_0}) = 0, \qquad \forall p \geq 2. \]

\noindent We now show that $P : LC_n \to X_{n,p_0}$ is surjective. Fix
$A_{n,p_0} \in \R^{\HH_n^{(p_0)}}$, and define for $p \geq 2$:
\[ A_{n,p} := \pi_{n, p_0 p \to p} \circ \pi_{n, p_0 p \to
p_0}^{-1}(A_{n,p_0}) \in \R^{\HH_n^{(p)}}. \]

\noindent It remains to show that $\pi_{n,kp \to p}(A_{n,kp}) = A_{n,p}$
for all $p \geq 2$ and $k \geq 1$. This follows by diagram chasing in the
following diagram, which commutes by Equation \eqref{Ecompatible}. 
\[
\xymatrix{
& & X_{n, kp p_0}
\ar[rr]^{\pi_{n,kp p_0 \to kp}}
\ar[d]|{\pi_{n,kp p_0 \to p_0}}
\ar[dll]_{\pi_{n,kp p_0 \to p_0}}
\ar[drr]^{\pi_{n,kp p_0 \to p}}
&& X_{n, kp}
\ar[d]^{\pi_{n,kp \to p}}\\
X_{n, p_0} & & X_{n,p p_0}
\ar[ll]^{\pi_{n, p p_0 \to p_0}}
\ar[rr]_{\pi_{n, p p_0 \to p}}
&& X_{n,p}
}
\]


\item Recall by $S_{[0,1]}$-invariance, that $\Lambda$ restricted to
$\sigma(\WW_{\bf p}^n)$ is determined by $C_n(\Lambda) =
(C_{n,p}(\Lambda))_{p \geq 2}$. In turn, $C_n(\Lambda) \in LC_n$ is
determined by $C_{n,p_0}(\Lambda) \in X_{n,p_0}$ for any $p_0 \geq 2n$,
by the previous part.
\end{enumerate}
\end{proof}

Though the main goal of this subsection was to prove Theorem
\ref{thm:upper} (along the way to proving Theorem \ref{thm:main}), 
a question of independent interest is to explicitly compute all entries
of the triangular matrix $\pi_{n,k}$. We conclude this part by providing
the solution to this question.

\begin{prop}\label{prop:pi_n_k}
Fix $n, k \in \N$. The entries of the matrix $\pi_{n,k} \in \Z_{\geq
0}^{\HH_n \times \HH_n}$ from Theorem \ref{thm:consistency} are given by:
\begin{equation*}
\pi_{n,k}([\novert{G}],[\novert{H}]) := \frac{1}{|\Aut(H)|}
\sum_{\psi: H \twoheadrightarrow G} \prod_{v \in V(G)} k(k-1)\cdots (k-
|\psi^{-1}(v)|+1),
\end{equation*}

\noindent where $G,H$ are arbitrary multigraphs with $n$ edges and no
isolated nodes.
\end{prop}

\begin{proof}
Consider fixed multigraphs $G,H$ without isolated vertices such that
$[\novert{G}], [\novert{H}] \in \HH_n$.  Denote $g = [\novert{G}]$
and $h = [\novert{H}]$.  Also fix $p \ge 2n$. Then by Proposition
\ref{prop:gamma}, there exists $x = (e^p_{(a_l,b_l)})_{l=1}^n \in E_p^n$
such that $[\novert{\Gamma}_{n,p}(x)] = g$. For the remainder of this
proof, we fix such an $x$ as well as an isomorphism $\phi: G \to
\novert{\Gamma}_{n,p}(x)$. 

Now define a map $y_x: \{1, \dots, k\}^{2n}
\to E_{kp}^n$ by:
\[ y_x((i_l,j_l)_{l=1}^n) := (e^{kp}_{(k(a_l-1) + i_l,k(b_l-1) +
j_l)})_{1 \leq l \leq n}. \]

\noindent In this new notation, our aim is to compute the quantity
$\pi_{n,k}(g,h)$, which is the number of tuples $(i_l,j_l)_{l=1}^n \in
\{1, \dots, k\}^{2n}$ such that
$[\novert{\Gamma}_{n,kp}(y_x((i_l,j_l)_{l=1}^n)))] = h$.  To do so,
recall from the proof of Theorem \ref{thm:consistency} that such a tuple
also yields a graph surjection from
$\Gamma_{n,kp}(y_x((i_l,j_l)_{l=1}^n)))$ to $\Gamma_{n,p}(x)$ as follows.

Given a tuple $(i_l,j_l)_{l=1}^n \in \{ 1, \dots, k \}^{2n}$, let $y =
y_x((i_l,j_l)_{l=1}^n) \in E_{kp}^n$.  As in the proof of Theorem
\ref{thm:consistency}, there exists a surjective map $\Gamma_{n,kp}(y)
\to \Gamma_{n,p}(x)$ defined by sending the vertex $v \in \{1, \dots,
kp\}$ to $\lfloor(v-1)/k\rfloor + 1 \in \{1, \dots p\}$ and sending the
$l$th edge $y(l) \in E(\Gamma_{n,kp}(y))$ to the $l$th edge $x(l) \in
E(\Gamma_{n,p}(x))$.  Composing this surjection with the map $\phi^{-1}$
yields a surjective map $\novert{\Gamma}_{n,kp}(y) \to G$ that we call
$P_y$.  Since the map $y_x$ is an injection there is no ambiguity in
using the $P_y$ notation.  

We would like to compute the quantity $\pi_{n,k}(g,h)$ by summing over
all possible surjective maps $\novert{\Gamma}_{n,kp}(y) \to G$ that arise
in the above manner.  However, to deal with the fact that the
$\Gamma_{n,kp}(y)$ are distinct graphs, we instead begin by associating
to surjective maps from a fixed graph $H$ to $G$ different tuples $y \in
\img(y_x)$ such that the surjective map can be factored through $P_y$.
More precisely, we define a map $A: \Surj(H,G) \to 2^{E_{kp}^n}$ as
follows. Given any surjective map of multigraphs $\psi: H \to G$, the set
$A(\psi) \subset E_{kp}^n$ will consist of the distinct $y \in \img(y_x)
\subset E_{kp}^n$ such that there exists an isomorphism $\beta: H \to
\novert{\Gamma}_{n,kp}(y)$ with $\psi = P_y \circ \beta$.

We would like to compute the number of $y \in \img(y_x)  \subset
E_{kp}^n$ such that $[\novert{\Gamma}_{n,p}(y)] = h$.  We claim that any
such $y$ is in $\cup_{\psi:H \twoheadrightarrow G}A(\psi)$.  Indeed, one
can take an arbitrary isomorphism $\beta^{-1}: \novert{\Gamma}_{n,p}(y)
\to H$, and then define $\psi : = P_y \circ \beta$.  Then $\psi: H \to G$
is a surjective map and $y \in A(\psi)$.
Conversely, if $y \in A(\psi)$ then by the above analysis,
$\Gamma_{n,p}(y)$ is isomorphic to $H$ through $\beta$ and $y \in
\img(y_x)$ by definition. Therefore, to compute $\pi_{n,k}$ it suffices
to compute the size of $\cup_{\psi:H \twoheadrightarrow G}A(\psi)$.  

To do this, we first show that $|A(\psi)| = \prod_{v \in V(G)}
k(k-1)\cdots (k- |\psi^{-1}(v)|+1)$. Indeed, for every $v \in V(G)$
consider the $k(k-1)\cdots (k- |\psi^{-1}(v)|+1)$ distinct ways of
sending the vertices of $u \in \psi^{-1}(v)$ to distinct vertices
$k(V_\psi(v) - 1) + i_u$ for $i_u \in \{1, \dots, k\}$.

Each choice defines an injective map $V_\beta: V(H) \to \{1, \dots,
kp\}$.  We can now pick a tuple $(i_l,j_l)_{l=1}^n \in \{1, \dots,
k\}^{2n}$ and $E_\beta: E(H) \to \{1, \dots, n\}$ such that
$(V_\beta(e_s), V_\beta(e_t)) = (y(E_\beta(e))_s, y(E_\beta(e))_t)$ where
$y = y_x((i_l,j_l)_{l=1}^n)$.  Then $V_\beta$ and $E_\beta$ together
define an isomorphism $\beta: H \to \novert{\Gamma}_{n,kp}(y)$ such that
$\psi = P_y \circ \beta$.  It is not hard to see that every element of $y
\in A_{\psi}$ arises in this way.

To finish the proof we show that counting every $y \in A_\psi$ overcounts
by a factor of $|\Aut(H)|$. Note that each automorphism $\alpha \in
\Aut(H)$ yields a distinct surjective map $\psi \circ \alpha: H \to G$
and $A(\psi) = A(\psi \circ \alpha)$.  Conversely, assume that there
exists $y \in E_{kp}^n$ such that $y \in A(\psi_1) \cap A(\psi_2)$.
Then there are isomorphisms $\beta_i: H \to \novert{\Gamma}_{n,kp}(y)$
satisfying $\psi_i = P_y \circ \beta_i$. Therefore $\psi_1  \circ
\beta_1^{-1} \circ \beta_2 = \psi_2$, and $\beta_1^{-1} \circ \beta_2 \in
\Aut(H)$.  We conclude that the sets $A(\psi)$ are either disjoint or
equal.  In addition, $A(\psi_1) = A(\psi_2)$ if and only if there exists
$\alpha \in \Aut(H)$ such that $\psi_1 = \psi_2 \circ \alpha$.
The result now follows.
\comment{
Therefore we finally obtain:
\begin{equation*}
\pi_{n,k}([\novert{G}],[\novert{H}]) := \frac{1}{|\Aut(H)|}
\sum_{\psi: H \twoheadrightarrow G} \prod_{v \in V(G)} k(k-1)\cdots (k-
|\psi^{-1}(v)|+1).
\end{equation*}
}
\end{proof}

\begin{remark}
In certain special cases, the formula of Proposition \ref{prop:pi_n_k} is
easy to evaluate. For instance when $[\novert{G}]=[\novert{H}]$, we
have
\begin{equation*}
\pi_{n, k}([\novert{G}],[\novert{H}]) = \frac{1}{|\Aut(G)|}
\sum_{\psi: H \twoheadrightarrow G} \prod_{v \in V(G)} k(k-1)\cdots (k-
|\psi^{-1}(v)|+1).
\end{equation*} 
However, $|\psi^{-1}(v)| = 1$ always in this case and $|\Surj(H,G)| =
|\Aut(G)|$, so this formula reduces to:
$\pi_{n,k}([\novert{G}],[\novert{H}]) = k^{|V(G)|}$.

If instead $[\novert{H}] = A_2^{\coprod n}$, then
\comment{
\begin{equation*}
\pi_{n, k}([\novert{G}],  A_2^{\coprod n})= \frac{1}{|\Aut(H)|}
\sum_{\psi: H \twoheadrightarrow G} \prod_{v \in V(G)} k(k-1)\cdots (k-
|\psi^{-1}(v)|+1).
\end{equation*} 
In this case,
}
$|\Aut(H)| = 2^n n!$ and each of the $2^n n!$ maps $\psi \in \Surj(H,G)$
satisfy $|\psi^{-1}(v)| = \deg(v)$, so we obtain:
$\pi_{n,k}([\novert{G}], A_2^{\coprod n}) = \prod_{v \in V(G)} k(k-1)
\cdots (k-\deg(v)+1)$.
\end{remark}

\subsection{Bases of consistent vectors}\label{section:mainproof}

Given a $C^n$ class function $F$, note by Proposition
\ref{prop:Cn_consistent} that $d^nF(0;g_1, \dots, g_n) \in \lie{X}_n$.
Now define
\begin{equation*}
T_{n,p}(F)(h) := C_{n,p}(d^n F(0,-))(h), \qquad p \geq 2, \ h \in
\HH_n^{(p)}.
\end{equation*}

\noindent By Theorem \ref{thm:consistency}, we obtain a linearly
consistent vector $T_n(F) = (T_{n,p}(F))_{p \geq 2} \in \prod_{p \geq 2}
X_{n,p}$, where $T_{n,p}(F) := (T_{n,p}(F)(h))_{h \in \HH_n^{(p)}}$.
Define $T_0(F) := F(0)$.  If $F$ is a $C^N$ class function, then $T_n(F)$
is defined for integers $n \in \{0, \dots, N\}$. We will sometimes write
\begin{equation*}
T(F) := (T_n(F))_{n \ge 0} \in \prod_{n \ge 0} \prod_{p \ge 2} X_{n,p}
\end{equation*}
for the entire collection if $F$ is smooth.  

Theorem \ref{thm:upper} then asserts that the values of the derivatives
$d^nF(0;g_1, \dots, g_n)$ for directions $g_1, \dots, g_n \in \WW_{\bf
p}$ are determined by $T(F)$.  In the case of solutions $F$ to the
differential equation of Theorem \ref{thm:main}, $T(F)$ determines $F$
restricted to $\WW_{\bf p}$ -- and by continuity, on all of $\W$.
Therefore, Theorem \ref{thm:upper} in fact already shows that the space
of solutions in Theorem \ref{thm:main} has dimension at most $\sum_{n \le
N} |\HH_n|$.

To complete the proof we now show that the $t(H,-)$ form a linearly
independent family of solutions. In order to do so, we obtain a general
formula for the derivatives of $t(H,-)$ in Proposition
\ref{prop:highergateaux}. We then use that result in Theorem
\ref{thm:hom_prop} to give a combinatorial formula for the $T(t(H,-))$
from which linear independence follows.  

We begin by stating the following useful lemma, whose proof is
straightforward after extending $\Lambda$ to its unique multilinear 
extension on $\WW^n$.

\begin{lemma}\label{lem:multilinear_deriv}
Let $\Lambda: \W^n \to \R$ be a multilinear functional.  Then the
function $F(g) := \Lambda(g,g,\dots, g)$ is G\^ateaux smooth. The
G\^ateaux derivatives of $F$ are
\[ d^kF(f;g_1,g_2, \dots, g_k) = \frac{1}{(n-k)!}\sum_{\sigma \in S_n}
\Lambda (\sigma(f,\dots, f, g_1, \dots, g_k)) \qquad \forall 0 \leq k
\leq n, \]
where $f \in \W$, $g_i \in \Adm(f)$, and $(f,\dots, f, g_1, \dots, g_k)
\in \WW^n$.
\end{lemma}

\comment{
\begin{proof}
By multilinearity,
\begin{align*}
dF(f;g_1) &= \lim_{h \to 0} \frac{\Lambda(f + hg_1,f+hg_1,\dots, f +
hg_1) - \Lambda(f,f,\dots, f) }{h}\\ 
&= \lim_{h \to 0} \left[\Lambda(g_1, f,\dots, f) +  \Lambda(f, g_1, f,
\dots, f) + \cdots +  \Lambda(f,\dots, f,g_1) + O(h) \right]\\
&=\Lambda(g_1, f,\dots, f) +  \Lambda(f, g_1, f, \dots, f) + \cdots +
\Lambda(f,\dots, f,g_1)\\
&=\frac{1}{(n-1)!}\sum_{\sigma \in S_n} \Lambda (\sigma(f,\dots, f,
g_1)).
\end{align*}
Note that each term in the sum is also multilinear so the result follows
by induction. Finally, all G\^ateaux derivatives of order $n+1$ are zero,
so all G\^ateaux derivatives of all orders are continuous.
\end{proof}
}

\noindent We now apply Lemma \ref{lem:multilinear_deriv} to compute
derivatives of homomorphism densities $t(H,f)$.

\begin{prop}\label{prop:highergateaux}
Given $H \in \HH_n$, the functions $t(H,-)$ are G\^ateaux smooth. Their
G\^ateaux derivatives $d^nt(H,f;g_1,g_2, \dots, g_n)$ are all zero if $n
> |E(H)|$, while if $n \leq |E(H)|$, then
\begin{align*}
&\ d^nt(H,f;g_1,g_2, \dots, g_n) \\
= &\ \sum_{\substack{A \subset E(H)\\ |A| = n}} \quad \sum_{\sigma: \{1,
\dots n\} \twoheadrightarrow A} \ \int_{[0,1]^{V(H)}} \prod_{l=1}^n
g_l(x_{\sigma(l)_s}, x_{\sigma(l)_t}) \prod_{e\in E(H) \backslash A}
f(x_{e_s}, x_{e_t}) \prod_{i \in V(H)}dx_i.
\end{align*}
\end{prop}

\begin{proof}
Define $\Lambda: \W^{|E(H)|} \to \R$ by
$\displaystyle \Lambda((f_e)_{e \in E(H)}) := \int_{[0,1]^{V(H)}}
\prod_{e \in E(H)} f_e(x_{e_s},x_{e_t}) \prod_{i \in V(H)} dx_i$.
Then $\Lambda$ is a multilinear functional with $\Lambda(f,f,\dots,f) =
t(H,f)$, so the result follows now from Lemma \ref{lem:multilinear_deriv}
applied to $\Lambda$.
\end{proof}

If $H \in \HH_{\le N}$, Proposition \ref{prop:highergateaux} shows that
the function $t(H,-)$ is G\^ateaux smooth and satisfies
$d^{N+1} t(H,f; g_1, \cdots, g_{N+1}) \equiv 0$
for $|E(H)| \leq N$ and all $g_1, \cdots, g_{N+1}$.
To prove Theorem \ref{thm:main}, we show that the space of linearly
consistent vectors in $\prod_{p \ge 2} X_{n,p}$ is spanned by
$(T_n(t(H,-)))_{H \in \HH_n}$.  Since there are exactly $| \HH_{\leq N}|$
of them, we only need to show that they are linearly independent.  This
linear independence follows from the following result, which proves a
formula relating the derivatives of $t(H,f)$ obtained above to
combinatorial quantities.  

\begin{theorem}\label{thm:hom_prop}\hfill
\begin{enumerate}[(i)]
\item Let $n \in \N$ and $p \ge 2$ be integers.  If $H \in \HH$ and $h
\in \HH_n^{(p)}$, then $T_{n,p}(t(H,-))(h) =0$ if $|E(H)| \neq n$.

\item Let $H \in \HH_n$ and $h \in \HH_{n}^{(p)}$.  Then
$T_{n,p}(t(H,-))(h) = |\Surj(H,\novert{h})|/p^{|V(H)|}$.
Therefore, $T_{n,p}(t(H,-))(h) > 0$ for $H \in \HH_n$ and $h \in
\HH_{n}^{(p)}$ if and only if there exists a surjective map from $H$ to
$\novert{h}$.

\item The vectors $(T_{n,p}(t(H,-))(h))_{h \in \HH_n^{(p)}}$ for fixed $p
\geq 2n$ and $H$ varying over all of $\HH_n$ are linearly independent.

\item For all $n \in \N$, the vector space $LC_n$ of linearly consistent
vectors $(A_{n,p} (h))_{h \in \HH_n, p \geq 2} \in \prod_{p \geq 2}
X_{n,p}$ has a basis given by
$\{ (T_n(t(H,-))(h))_{h \in \HH_n^{(p)}, p \ge 2} \ : \ H \in \HH_n \}$.
\end{enumerate}
\end{theorem}

\begin{proof}\hfill
\begin{enumerate}[(i)]
\item For $|E(h)| > |E(H)|$,  $T_{n,p}(t(H,-))(h) =0$ because higher
derivatives are zero by Proposition \ref{prop:highergateaux}.  For
$|E(h)| < |E(H)|$, $T_{n,p}(t(H,-))(h) =0$ by Proposition
\ref{prop:highergateaux}, because we are evaluating a lower-order
derivative at $0$.

\item We claim that $T_{n,p}(t(H,-))(h) = |\Surj(H,h)|/p^{|V(H)|}$ where
$\Surj(H,h)$ is the set of surjective maps from $H$ to $h$.  
To prove this claim, let $H \in \HH_n$ be fixed.  For every fixed $h \in
\HH_n^{(p)}$, fix a bijective map $\phi : V(h) \to \{1, \dots, p\}$. We
now define the tuples ${\bf g}(h)$ by: ${\bf g}(h)_e := e^p_{(\phi(e_s),
\phi(e_t))}$ for $e \in E(h)$. Then for all $H \in \HH_n$,
\begin{align}\label{EMHG3}
T_{n,p}(t(H,-))(h) & = d^n t(H,0;({\bf g}(h)_e)_{e \in E(h)}) \notag\\
&= \sum_{\sigma: E(H) \twoheadrightarrow E(h)} \int_{[0,1]^{V_0}}
\prod_{e \in E(H)} {\bf g}(h)_{\sigma(e)}(x_{e_s}, x_{e_t}) \prod_{i \in
V_0} dx_i,
\end{align}

\noindent where the last equality follows by Proposition
\ref{prop:highergateaux}. (Note that the order of G\^ateaux
differentiation does not matter since mixed partials are equal.)

To prove the claim, consider an arbitrary term in \eqref{EMHG3}.  
Then ${\bf g}(h)_e$ is constant on each ``sub-rectangle" in $[0,1]^{V(H)}$
of size $1/p^{|V(H)|}$. Hence
\begin{align}\label{EMHG4}
&\ \int_{[0,1]^{V(H)}} \prod_{e \in E(H)} {\bf g}(h)_{\sigma(e)}(x_{e_s},
x_{e_t}) \prod_{i \in V(H)} dx_i \notag\\
= &\ \frac{1}{p^{|V(H)|}} \sum_{\tau: V(H) \to \{1, \dots, p\}} \quad
\prod_{e \in E(H)} {\bf g}(h)_{\sigma(e)}\left(\frac{\tau(e_s)-0.5}{p},
\frac{\tau(e_t)- 0.5}{p}\right).
\end{align}

By our choice of ${\bf g}(h)_e$ we have
\begin{equation*}
{\bf g}(h)_{\sigma(e)}\left(\frac{\tau(e_s)- 0.5}{p}, \frac{\tau(e_t)-
0.5}{p}\right) =
\begin{cases}
1 & \text{ if } \{\tau(e_s), \tau(e_t)\} = \{ \phi(\sigma(e)_s),
\phi(\sigma(e)_t)\},\\
0 &\ \text{otherwise}.
\end{cases}
\end{equation*}

\noindent Since $\phi$ is bijective we can define for such $\tau$ vertex
maps $\phi^{-1}\tau: V(H) \to V(h)$. We recognize the expression on the
right hand side of Equation \eqref{EMHG4} to be equal to
$\frac{1}{p^{|V(H)|}}$ times the number of vertex maps $\phi^{-1}\tau:
V(H) \to V(h)$ that form a map of multigraphs $H \to h$, when combined
with the edge map $\sigma: E(H) \twoheadrightarrow E(h)$.  In addition,
Equation \eqref{EMHG3} sums over all surjective maps $\sigma: E(H)
\twoheadrightarrow E(h)$ so $T_{n,p}(t(H,-))(h) =
|\Surj(H,h)|/p^{|V(H)|}$.  Since $H$ has no isolated vertices, there is a
natural bijection between $\Surj(H,h)$ and $\Surj(H,\novert{h})$; thus
the result follows.

\item The vectors $(T_{n,p}(t(H,-))(h))_{h \in \HH_n}$ for fixed $p \geq
2n$ and $H$ varying over all of $\HH_n$ form an upper triangular matrix
with non-zero diagonal if ordered consistently with the existence of a
surjection. Therefore, the matrix is non-singular and the assertion of
linear independence follows.


\item The vector space of linearly consistent vectors $(A_{n,p}(h))_{h
\in \HH_{n}^{(p)}, p \geq 2}$ has dimension $|\HH_n|$ by Theorem
\ref{thm:upper}.  On the other hand, $T(t(H,-))$ for $H \in \HH_{n}$ are
a linearly independent sets of size $|\HH_n|$ so they form a basis.
\end{enumerate}
\end{proof}

As a consequence of Theorem \ref{thm:hom_prop}, we show the linear
independence of $t(H,-)$ for multigraphs $H \in \HH$.  

\begin{cor}\label{cor:whitney}
The homomorphism densities $t(H,-)$ for $H \in \HH$ are linearly
independent as functions on $\WW_{\bf p}:= \cup_{p = 1}^\infty \WW_p$,
and hence on $\WW$. In particular, the $t(H,-)$ are also linearly
independent for $H \in \G$.
\end{cor}

\begin{proof}
Assume that there is a finite linear relation $\sum_{H \in \HH} a_H
t(H,f) = 0$ (for all $f$). Now by taking derivatives at zero, it is clear
by Theorem \ref{thm:hom_prop}(i) that $\sum_{H \in \HH_n} a_H T_n(t(H,-))
= 0$ for every $n \in \N$. By Theorem \ref{thm:hom_prop}(iv), we conclude
that $a_H = 0$ for all $H \in \HH_n$ and all $n$.
\end{proof}

\begin{remark}
The proof shows that the linear independence of functions $t(H,-)$ for $H
\in \HH_{\leq n}$ holds even when restricted to their values on
$\WW_{2n}$.  Such linear independence results go back to Whitney
\cite{Whitney} for $H \in \G$.
A powerful result stated in \cite[Theorem 1]{ELS:1979} implies that there
can be no algebraic relations between homomorphism densities of connected
graphs.  We restate it in graphon language below.

\begin{theorem}[Erd\"os-Lov\'asz-Spencer, 1979]\label{thm:ELS}
Let $H_1, \dots H_m$ be all connected graphs (up to isomorphism) with
$|V(H_i)| \le k$.  
Then the image of the function $F: \W \to \R^m$ defined by $f \to
(t(H_1,f),t(H_2,f), \dots, t(H_m,f))$ contains an open ball.  
\end{theorem}
\end{remark}

We now bring together the results of the previous sections and this
section to prove the main theorem.

\begin{proof}[{\bf Proof of Theorem \ref{thm:main}}] 
We associated to any smooth continuous class solution $F$, the linearly
consistent vector $T(F) \in \prod_{n \ge 0} \prod_{p \ge 2} X_{n,p}$. By
parts (i),(iv) of Theorem \ref{thm:hom_prop}, the space of linearly
consistent vectors are spanned by those arising from homomorphism
densities.  It follows that there exist constants $a_H$ so that $T(F)$
can be written as
\begin{equation}\label{Etaylorexpansion}
T(F) = \sum_{H \in \HH_{\le N}} a_H T(t(H,-)).
\end{equation}

Note that for any fixed direction $f \in \WW_{\bf p}$, the
one-dimensional differential equation is solved by
\begin{equation*}
F(f) = F(0) + dF(0;f) + \frac{d^2F(0;f,f)}{2!} + \dots +
\frac{d^N F(0;f,\dots,f)}{N!}.
\end{equation*}

\noindent Since the derivatives are all multilinear, the value of $F$ on
$\WW_{\bf p}$ is therefore determined by $T(F)$ by this formula.  The
same applies to $t(H,-)$ for all $H \in \HH_{\leq N}$, so we see that
$F(f) = \sum_{H \in \HH_{\le N}} a_H t(H,f)$ on $f \in \WW_{\bf p}$.  
 
Since both sides are continuous in the $L^1$ topology by Proposition
\ref{prop:cont_func_char}, and since $\WW_{\bf p}$ is dense in $\W$ in the
$L^1$ topology, the two sides are equal on all $f \in \W$. In addition,
by Corollary \ref{cor:whitney} the coefficients $a_H$ are unique. This 
shows the first part of the result. If we assume further that $F$ is
continuous with respect to the cut-norm, then Proposition
\ref{prop:cont_func_char} shows that
$F(f) = \sum_{H \in \HH_{\le N}} a_H t(H^{simp},f)$
where $H^{simp}$ is the simple graph obtained from $H$ by retaining only
a single edge between two vertices of $H$ if there are one or more edges
connecting the same two vertices.

\noindent On the other hand, $F(f) = \sum_{H \in \HH_{\leq N}} a_H
t(H,f)$ for all $f \in \W$ from above. Using Corollary \ref{cor:whitney},
this is possible if and only if $a_H = 0$ for all $H \in \HH_{\leq N}
\setminus \G_{\leq N}$, and the second part of the result follows.
\end{proof}

\begin{remark}
Note that it is enough in the statement of the theorem to only assume
that
\[ d^{N+1}(F)(cg; g, g, \dots, g) \equiv 0, \qquad \forall g \in \WW_{\bf
p}, \ c \in (0,1), \]

\noindent since the above proof only uses this assumption.

Moreover, the first part of Theorem \ref{thm:main} holds for any topology
on $\WW$ such that $\WW_{\bf p}$ is dense in $\WW$ and such that $t(H,-)$
for $H \in \HH$ is continuous with respect to the topology.
\end{remark}

\subsection{Characterizing homomorphism densities}

As we have seen, homomorphism densities $t(H,-)$ are fundamental
continuous class functions on $\W$.  It is natural to characterize
homomorphism densities amongst all continuous class functions.  Indeed,
similar characterizations exist (\cite[Section 5.6]{Lovasz:2013}) for
functions $F : \G \to \R$  of the form $Hom(-,H)$ and $Hom(H,-)$. 

Our characterization is based on the work of the previous sections and
the notion of the tensor product of two graphons.  Recall from
\cite[Section 7.4]{Lovasz:2013} that given graphons $f,g \in \WW$, their
\textit{tensor product} $f \otimes g : [0,1]^4 \to [0,1]$ is defined to
be the map:
\[ (f \otimes g) (x_1,x_2,y_1,y_2) := f(x_1,y_1) g(x_2,y_2). \]

\noindent Given an arbitrary measure preserving map $\phi: [0,1] \to
[0,1]^2$, the tensor product can be associated with
the following graphon:
$(f \otimes g)^\phi (x,y) := (f \otimes g) (\phi(x),\phi(y))$.
Different choices of $\phi$ yield weakly equivalent graphons. With a
small abuse of notation, we shall write $f \otimes g$ to denote the
graphon $(f \otimes g)^\phi$, where $\phi : [0,1] \to [0,1]^2$ is fixed
for the rest of this section.

The tensor product has the property (\cite[Section 7.4]{Lovasz:2013})
that for all multigraphs $H$ and graphons $f,g$,
\begin{equation}\label{eq:tensor_mult}
t(H,f \otimes g) = t(H,f) t(H,g).
\end{equation}

\noindent In particular, for every $m \in \N$ and multigraph $H \in \HH$,
we have: $t(H,f^{\otimes m}) = t(H,f)^m$.

We now characterize homomorphism densities in terms of tensor products
and G\^ateaux derivatives.

\begin{theorem}\label{thm:characterization}
A class function $F: \W \to \mathbb{R}$ is a homomorphism density
$t(H,-)$ for $H \in \HH$ (resp.~$H \in \G$), if and only if it satisfies
the following properties:
\begin{enumerate}[(i)]
\item $F$ is continuous in the $L^1$-topology (resp.~cut topology) on
$\W$;
\item there exists $n > 0$ so that $d^{n} F(f; g_1, \dots, g_n) = 0$ for
all $f \in \W$ and $g_1, \dots, g_n \in \Adm(f)$; and
\item $F(f)^m = F(f^{\otimes m})$ for some even integer $m \in \N$, for
all $f \in \W$.
\end{enumerate}
\end{theorem}

\begin{proof}
Suppose first that $F$ is of the form $F(f) = \sum_{i=1}^n a_i t(H_i,f)$
with $a_i \neq 0$ for all $i$ and $H_i \in \HH$ pairwise distinct.
To prove the result in this case, we define a well-ordering $\le$ on
$\HH$ as follows. First define a well-ordering on the set of connected
multigraphs $H \in \HH$ by setting $H < G$ if $|E(H)| < |E(G)|$, and
picking an arbitrary total ordering on the finite set of connected
multigraphs with a fixed number of edges.  Now any finite multigraph has
a unique decomposition into finitely many connected multigraphs of the
form $H = \coprod_{i=1}^l H_i^{\coprod k_i}$.  The set $\HH$ can
therefore be put into bijection with sequences of non-negative integers
$e_H$, where $H$ ranges over connected finite multigraphs and only
finitely many $e_H$ are non-zero.
Therefore $\HH$ can be ordered lexicographically since we have already
put a total order on the connected multigraphs in $\HH$.  This ordering
on $\HH$ is now a well-ordering, whose unique minimum is given by the
graph with a single vertex and no edges. It satisfies the property that
if $H_1 \le H_2$ and $G_1 \le G_2$, then $H_1 \coprod H_2 \le G_1 \coprod
G_2$, with equality if and only if $H_1 = H_2$ and $G_1 = G_2$.  

Now by assumption, 
\[ \left[\sum_{i=1}^n a_i t(H_i,f)\right]^m = \sum_{i=1}^n a_i
(t(H_i,f))^m \]
for some even $m >1$.  By the linear independence of multigraph
homomorphism densities (Corollary \ref{cor:whitney}) we must have
equality termwise in $\HH$. Assume without loss of generality that $H_1 <
\cdots < H_n$ in the total ordering $<$, and suppose for contradiction
that $n>1$. Then the cross term $m a_{1}^{m-1} a_{2} t(H_1^{m-1} \coprod
H_2, f)$ on the left hand side is non-zero but does not appear on the
right hand side.  We conclude that there is at most one non-zero $a_i$.
In that case, $a_i = a_i^m$ so $a_i \in \{0,1\}$.

Finally, suppose $F$ is any function satisfying the hypotheses. By
conditions (i) and (ii) and Theorem \ref{thm:main}, $F$ is of the form
$F(f) = \sum_{i=1}^k a_i t(H_i,f)$ with $H_i \in \HH$ or $\G$ for all
$i$, depending on the topology used in (i). Therefore the theorem follows
from the above analysis.
\end{proof}

\begin{remark}
We now explain how condition (iii) above can be replaced by a purely
combinatorial condition. Note that by the assumed continuity of $F$ in
the cut-norm, $F(f)^m$ is continuous in $f$.  Also note that if $f_k \to
f$ then $t(H,f_k^{\otimes m} ) \to t(H,f^{\otimes m})$ for every simple
graph $H$ by Equation \eqref{eq:tensor_mult}. Thus $f_k^{\otimes m} \to
f^{\otimes m}$ -- i.e., $f \mapsto f^{\otimes m}$ is continuous -- and
hence, $F(f^{\otimes m})$ is continuous with respect to $f$ as well.  
Therefore, it suffices to assume condition (iii) for simple graphs $f^H$
in place of general $f$. Moreover, the tensor product of two graphs
$f^{H_1} \otimes f^{H_2}$ is weakly equivalent to a graph(on)
corresponding to the Kronecker product of the adjacency matrices of $H_1$
and $H_2$. Thus, the third condition can be replaced by a purely
combinatorial condition involving only finite simple graphs -- i.e.,
Kronecker powers of their adjacency matrices.
\end{remark}

\comment{
\begin{remark}
A natural question that now arises is whether all class functions $F :
\WW \to \R$ that are multiplicative (with respect to $\otimes$) are of
the form $t(H,-)$ for some $H \in \HH$. The answer turns out to be
negative. For instance, consider a finite set of graphs $S := \{ H_1,
..., H_k \} \subset \G$, and define
$M_S(f) := \max_{1 \leq i \leq k} t(H_i, f)$.
This is clearly multiplicative, but we claim that it is not
$C^1$ if there are two graphs in $S$ with unequal number of edges. For in
this case, consider the behavior of $M_S$ along the ``constant" line $f
\equiv c$. Then $t(H_i, c) = c^{|E(H_i)|}$ for all $i$. So below $c=1$,
the function is $M_S(c) = c^{\min |E(H_i)|}$ while above $1$, it equals
$c^{\max |E(H_i)|}$. Hence $d M_S(1;1)$ does not exist. This
counterexample shows us why additional hypotheses are required in order
to characterize homomorphism densities.
\end{remark}
}

\section{Power series and Taylor series}\label{Sseries}

The previous sections demonstrate that $t(H,-)$ with $|E(H)| = n$ can be
seen as the analogue of monomials of degree $n$ in the graphon space. By
analogy to single variable Taylor series, we study expansion of smooth
class functions on $\W$ in terms of infinite series of homomorphism
densities.
We give sufficient conditions for the existence of such series in Section
\ref{Staylor}, and prove their uniqueness in Section \ref{Suniqueness}.
The proofs of these results rely heavily on the work of Section \ref{S4}.
In particular, in Section \ref{Staylor} we use Theorem \ref{thm:main} to
show that the Taylor expansion of a smooth class function can be written
in terms of homomorphism densities.  In Section \ref{Suniqueness} we
generalize the linear independence result of Section \ref{S4} to prove
uniqueness and explain the general philosophy behind the proofs of both
the linear independence results.

In order to do this we first investigate general facts about
differentiation and convergence of series of homomorphism densities in
Section \ref{Subseries}. The general theory also includes an algebra
structure on such series, that is obtained from the algebra structure on
formal linear combinations of graphs $\Q_0$.  The algebra structure can
be extend to formal linear combinations of $k$-labelled multigraphs
$\Q_k$, so we develop the general properties to include their homomorphism
densities.  

In Section \ref{SqN} we explain how to package all of the $\Q_k$ into a
single algebra $\Q_\N$ and then define {\it weighted homomorphism
densities}. These simultaneously generalize multigraph homomorphism
densities, partially labelled graph homomorphism densities, and their
derivatives.  We develop the theory of series in this generality in order
that the space of series is closed under differentiation and has an
algebra structure.

Along the way, we address in Section \ref{Subseries} one of Lov\'asz's
questions \cite[Problem 16]{Lovasz:Open} about whether it is possible to
expand right homomorphism densities in terms of left homomorphism
densities.  Finally, in Section \ref{Sinfq} we apply this theory to give
a proposal for an analytic theory of infinite quantum algebras, thereby
addressing another of Lov\'asz's questions \cite[Problem 7]{Lovasz:Open}.
This last application gives a second motivation for developing the
properties of series in the generality we do.

\subsection{The algebra of partially labelled multigraphs}\label{SqN}

In this section we recall how to equip the space of formal linear
combinations of partially labelled multigraphs with an algebra structure.
In addition, we define generalizations of homomorphism densities indexed
by such graphs with weights so that the functions are closed under
differentiation.  These functions will form the individual terms of the
infinite series we investigate in this paper.

\begin{definition}\label{Dgraph-map}
Given an integer $k \geq 0$, a \textit{$k$-labelled multigraph} is a
multigraph $H$ with an injective label map $l_H : \{1,\dots, k\} \to
V(H)$.   (If $k=0$ then $H$ is unlabelled.)
If $G,H$ are $k$-labelled multigraphs for $k>0$, then a map $f$ of such
multigraphs is defined by the data of a map
of vertices $V_f : V(H) \to V(G)$ and a map of edges $E_f : E(H) \to
E(G)$.  The maps $E_f$ and $V_f$ must be compatible in the sense that
\begin{equation*}
\{E_f(e)_s,E_f(e)_t\} = \{V_f(e_s),V_f(e_t)\} \ \forall e \in E(H),
\end{equation*}
and $V_f(l_H(a)) = l_G(a)$ for $a \in \{1, \dots, k\}$.

For all integers $k,n \geq 0$, define $\HH_{n,k}$ to be the finite set of
isomorphism classes of $k$-labelled multigraphs with $n$ edges and no
unlabelled isolated nodes. Finally, following \cite[Section
6.1]{Lovasz:2013}, define
\[ \Q_k := {\rm span}_\R \bigcup_{n \geq 0} \HH_{n,k} \]

\noindent to be the vector space with basis given by the disjoint union
of the sets $\{ \HH_{n,k} : n \geq 0 \}$.
\end{definition}

Note that $\Q_k$ can be given the structure of a commutative algebra by
defining the product of two labelled multigraphs $F_1$ and $F_2$ to be
the multigraph obtained by taking their disjoint union and identifying
equivalently labelled nodes.
In fact, we note that for all $n$, the set of graphs $\HH_{n,k}$ embeds
into $\HH_{n,k+1}$ by attaching an additional isolated node labelled
$k+1$ to any graph $H \in \HH_{n,k}$. This induces the obvious linear
injection $: \Q_k \hookrightarrow \Q_{k+1}$ for all $k$. Thus, define
(following \cite[Section 6.1]{Lovasz:2013}) their directed limit to be
the {\it algebra of partially labelled multigraphs}:
\[ \Q_\N := \bigcup_{k \geq 0} \Q_k = {\rm span}_\R \HH', \qquad \HH' :=
\bigcup_{k \geq 0} \coprod_{n \geq 0} \HH_{n,k}. \]

The space $\Q_0 = {\rm span}_\R \HH$ is called the space of {\it quantum
graphs} \cite[Section 6.1]{Lovasz:2013}. There is an algebra map $\alpha$
from $\Q_0$ to the space of class functions on $\W$, defined by
$\alpha(H) := t(H,-)$ and extending linearly.  This is a map of algebras
because the product on $\Q_0$ is just disjoint union and
$t(H_1,f)t(H_2,f) = t(H_1 \coprod H_2, f)$.

The map $\alpha$ has been generalized to define homomorphism densities
from partially labelled multigraphs $H \in \bigcup_{n \geq 0} \HH_{n,k}$
for any $k \geq 0$ -- see \cite[Section 7.2]{Lovasz:2013}. In the
language of this paper, fix $H \in \HH_{n,k}$ and ${\bf x}_1, \dots, {\bf
x}_k \in [0,1]$, and let $V_0 := V \setminus l_H( \{ 1, \dots, k \})$
denote the set of unlabelled vertices. We now define
\begin{equation}\label{Eproduct3}
f \mapsto t_{\bf x}(H,f) = t_{({\bf x}_1, \dots, {\bf x}_k)}(H,f) :=
\int_{[0,1]^{V_0}} \prod_{e \in E(H)} f(x_{e_s}, x_{e_t})\ \prod_{i \in
V_0} d x_i
\end{equation}
where $x_v = {\bf x}_{l_H^{-1}(v)}$ for labelled vertices $v \in V(H)
\setminus V_0$.  It is not difficult to see that this notion only depends
on the isomorphism class of $H$.

As in the $k=0$ case, we have
\begin{equation}\label{Eproduct2}
t_{\bf x}(H_1,f)t_{\bf x}(H_2,f) = t_{\bf x}(H_1 H_2, f)
\end{equation}

\noindent for all $k$-labelled multigraphs $H_1, H_2$. Hence $\alpha_{\bf
x} : \Q_k \to Func(\WW,\R)$ is an algebra map for any choice of ${\bf x}
\in [0,1]^k$.
More precisely, fix ${\bf x}_n \in [0,1]$ for all $n \in \N$, and define
${\bf x} := ({\bf x}_n) \in [0,1]^\N$. Then one can define the map
$\alpha_{\bf x} : \Q_\N \to Func(\WW,\R)$, given by
\begin{equation}
\alpha_{\bf x}(H) := t_{\bf x}(H,-), \qquad \forall {\bf x} \in
[0,1]^\N,\ H \in \bigcup_{k \geq 0} \bigcup_{n \geq 0} \HH_{n,k},
\end{equation}

\noindent and extending by linearity. Here if $H \in \HH_{n,k}$ then we
define $t_{\bf x}(H,-) := t_{({\bf x}_1, \dots, {\bf x}_k)}(H,-)$. The
following result is then immediate.

\begin{lemma}\label{Lalghom}
For all ${\bf x} \in [0,1]^\N$, 
$\alpha_{\bf x} : \Q_\N \to Func(\WW,\R)$ is an algebra homomorphism.
\end{lemma}

\noindent {\bf Weighted homomorphism densities.}

The space of functions $\scrs_t$ can be extended in several different
directions. It is the aim of this section to extend it to contain
multigraph homomorphism densities, the image of $\alpha_{\bf x}$,
infinite convergent series, and weighted series that arise naturally
after taking derivatives. To that end, we make the following definition.

\begin{definition}
Given a multigraph $H \in \HH_{n,k}$ with unlabelled nodes given by $V_0
:= V(H) \setminus l_H(\{ 1, \dots, k \})$, together with numbers ${\bf
x}_1, \dots, {\bf x}_k \in [0,1]$, and a bounded kernel $g_H :
[0,1]^{V(H)} \to \R$, define the corresponding {\it weighted partially
labelled multigraph homomorphism density} to be:
\begin{equation}
t_{({\bf x}_1, \dots, {\bf x}_k)}(H,f; g_H) := \int_{[0,1]^{V_0}}
g_H((x_i)_{i \in V(H)}) \prod_{e \in E(H)} f(x_{e_s}, x_{e_t})\ \prod_{i
\in V_0} d x_i.
\end{equation}
where $x_v = {\bf x}_{l_H^{-1}(v)}$ for labelled vertices $v \in V(H)
\setminus V_0$.
\end{definition}

\noindent In particular, $t_{\bf x}(H,f,1) = t_{\bf x}(H,f)$ in the
notation of \eqref{Eproduct3}. We use this without further reference in
the remainder of the paper.

We now generalize Equations \eqref{Eproduct} and \eqref{Eproduct2} to
weighted, partially labelled multigraphs. The following result is not
hard to show.

\begin{prop}\label{prop:dens_multiplicative}
We have for $H \in \HH_{n,k}$ and $H' \in \HH_{m,k}$:
\begin{equation*}
t_{\bf x}(H,f,g_H)t_{\bf x}(H',f,g'_{H'}) = t_{\bf x}(HH',f,g_H g'_{H'}). 
\end{equation*}
\end{prop}

\noindent Here, $g_H g'_{H'} : [0,1]^{V(HH')} \to \R$ denotes the
function $g_H((x_i)_{i \in V(H)}) g'_{H'}((x_i)_{i \in V(H')})$.

\comment{
\begin{proof}
Let $V_0$ and $V'_0$ be the unlabelled vertices of $H$ and $H'$
respectively. Then,
\begin{align*}
&\ t_{\bf x}(H,f,g_H)t_{\bf x}(H',f,g'_{H'})\\
= &\ \left( \int_{[0,1]^{V_0}} g_H((x_i)_{i \in V(H)}) \prod_{e \in E(H)}
f(x_{e_s}, x_{e_t})\ \prod_{i \in V_0} d x_i \right) \ \times\\
& \quad \left( \int_{[0,1]^{V'_0}} g'_{H'}((x_i)_{i \in V(H')}) \prod_{e
\in E(H')} f(x_{e_s}, x_{e_t})\ \prod_{i \in V'_0} d x_i \right)\\
= &\ \int_{[0,1]^{V_0} \times [0,1]^{V'_0}} g_H((x_i)_{i \in V(H)})
g'_{H'}((x_i)_{i \in V(H')}) \prod_{e \in E(HH')} f(x_{e_s}, x_{e_t})\
\prod_{i \in V_0} d x_i \prod_{i \in V'_0} d x_i \\
= &\ t_{\bf x}(HH',f,g_H g'_{H'}).
\end{align*}
\end{proof}
}

\begin{remark}
One can take (higher) G\^ateaux derivatives of $t_{\bf x}(H,f,g_H)$ along
arbitrary directions, just as for ordinary homomorphism densities
$t(H,f)$. More precisely,
\begin{align}\label{weighted-der}
&\ d (t_{\bf x}(H,f,g_H);g)\notag\\
= &\ \sum_{e_1 \in E(H)} \int_{[0,1]^{V_0}} g_H((x_i)_{i \in V(H)})
g(x_{(e_1)_s},x_{(e_1)_t})\prod_{e \in E(H) \setminus \{ e_1 \}}
f(x_{e_s}, x_{e_t})\ \prod_{i \in V_0} d x_i.\notag\\
= & \sum_{e_1 \in E(H)} t_{\bf x}(H_{e_1},f;g_{e_1}),
\end{align}

\noindent where:
\begin{itemize}
\item $H_{e_1} \in \HH_{n-1,k}$ is obtained from $H \in \HH_{n,k}$ by
removing the edge $e_1 \in E(H)$ and then further removing all unlabelled
isolated nodes. \item The new weights are:
\[ g_{e_1}((x_i)_{i \in V(H_{e_1})}) := \int_{[0,1]^{V(H) \setminus
V(H_{e_1})}} g_H((x_i)_{i \in V(H)}) g(x_{(e_1)_s}, x_{(e_1)_t})\
\prod_{i \in V(H) \setminus V(H_{e_1})} d x_i. \]
\end{itemize}

Equation \eqref{weighted-der} implies that the linear span of $t_{\bf
x}(H,-,g_H)$ is closed under differentiation, whereas the span of the
unweighted (partially labelled multigraph) homomorphism densities is not.
Further note that $t_{\bf x}(H,-,g_H)$ is $C^0$ in the sense of
Definition \ref{DCnatf}. Therefore the functions $t_{\bf x}(H,-,g_H)$ are
G\^ateaux smooth.
\end{remark}

\subsection{Series}\label{Subseries}

In the previous subsection, we defined the weighted homomorphism
densities, simultaneously generalizing simple graph homomorphism
densities, multigraph homomorphism densities, and partially labelled
graph homomorphism densities.  In addition we found that the span of such
functions is closed under addition, multiplication, and differentiation.
In this subsection we develop the basic convergence properties of
infinite series of such functions. We show that such ``absolutely
convergent" series are closed under addition, multiplication, and
differentiation.  In addition, we use this theory to examine whether
right homomorphism densities can be expanded as formal series of
homomorphism densities.

\begin{definition}\label{Dseries}
For $k \ge 0$ an integer and ${\bf x} = ({\bf x}_1, \dots, {\bf x}_k) \in
[0,1]^k$, define the {\it weighted $(k, {\bf x})$-labelled power series}
to be the set of formal series of the form
\[ \sum_{n=0}^\infty \sum_{H \in \HH_{n,k}} t_{\bf x}(H,-,g_H). \]

\noindent The subset of {\it $(k, {\bf x})$-labelled power series}
consists of those formal series for which $g_H$ is constant for all $H$.
Given ${\bf x} \in [0,1]^\N$, a {\it (weighted) ${\bf x}$-labelled power
series} is a (weighted) $(k, ({\bf x}_1, \dots, {\bf x}_k))$-labelled
power series for some $k \geq 0$.
\end{definition}

We first discuss the structure of the set of such series.
The following result is straightforward.

\begin{prop}\label{Palgebra}
Given ${\bf x} \in [0,1]^\N$, the set of weighted ${\bf x}$-labelled
power series is a (unital) commutative graded $\R$-algebra, with termwise
addition, and with multiplication given by:
\[ (FG)(-) := \ \sum_{n \geq 0} \sum_{H \in \HH_{n,\max(k,k')}}  t_{\bf
x} (H,-, \sum_{(H_1, H_2) : H_1 H_2 = H} f_{H_1} g_{H_2}), \]

\noindent where $F(-) = \sum_{n \geq 0} \sum_{H \in \HH_{n,k}} t_{\bf
x}(H,-,f_H)$ and $G(-) = \sum_{n \geq 0} \sum_{H \in \HH_{n,k'}} t_{\bf
x}(H,-,g_H)$. The weighted $(k, {\bf x})$-labelled power series form an
increasing family of subalgebras in $k \geq 0$ with the same properties.
\end{prop}

\noindent Note that the subspaces of unweighted power series form
subalgebras of the above algebras. Also note that in defining the
product, with a slight abuse of notation, we continue to denote the
weightings for $\max(k,k')$ by $f_{H_1}$ and $g_{H_2}$, but it is clear
what this means.

\comment{
\begin{proof}
We first discuss the multiplication operation. It is not hard to show
that the binary operations on functions and on $\HH'$, given respectively
by
$(g_H, g'_{H'}) \mapsto g_H g'_{H'}$ and $(H,H') \mapsto HH'$,
are each associative as well as commutative. This easily shows the
commutativity of the product. For associativity, given formal power
series $E,F,G$ with the obvious formal expansions, we compute using the
associativity of the aforementioned binary operations
that $(EF)G$ and $E(FG)$ both equal
\[ \sum_{n \geq 0} \sum_{H \in \HH_{n,\max(k_E,k_F,k_G)}} t_{\bf x} (H,-,
\sum_{(H_1, H_2, H_3) : H_1 H_2 H_3 = H} e_{H_1} f_{H_2} g_{H_3}), \]

\noindent where $k_E, k_F, k_G$ are part of the data of $E,F,G$
respectively.
Next, the addition, distributivity, existence of a unit $U$, and
$\R$-algebra structure are also easy to show. The unit is given by
$O_{\bf x} := t_{\bf x}(H,-,1)$, where $H$ is the unique graph in
$\HH_{0,k} = \HH_{0,0}$. Finally, the $n$th graded piece is precisely the
span of all $t_{\bf x}(H,-,g_H)$ for $H \in \HH_{n,k}$ (for all $k \geq
0$) and all $g_H$. The corresponding results for $(k, {\bf x})$-labelled
series now follow easily.
\end{proof}
}

We now study the convergence properties of formal power series. Given
Proposition \ref{Palgebra}, it suffices to study weighted $(k, {\bf
x})$-labelled power series for any fixed $k \geq 0$. Using the analogy to
monomials discussed in Section \ref{Sprelim}, we arrange weighted
$(k,{\bf x})$-labelled power series according to their ``degree", and say
that such a series \textit{converges at $f \in \WW$} if
\begin{equation*}
\sum_{n=0}^\infty \sum_{H \in \HH_{n,k}} t_{\bf x}(H,f,g_H) := \lim_{N
\to \infty} \sum_{n=0}^N \sum_{H \in \HH_{n,k}} t_{\bf x}(H,f,g_H)
\end{equation*}

\noindent exists. Similarly, a weighted $(k,{\bf x})$-labelled power
series \textit{converges absolutely at $f \in \WW$} if 
\begin{equation}\label{Eabsconv}
\sum_{n=0}^\infty \sum_{H \in \HH_{n,k}}t_{\bf x}(H,|f|,|g_{H}|) <
\infty.
\end{equation}

\noindent As for power series of one variable (in real analysis), we will
interchangeably use $F(-) = \sum_{n=0}^\infty \sum_{H \in \HH_{n,k}}
t_{\bf x}(H,-,g_H)$ and $F(f) = \sum_{n=0}^\infty \sum_{H \in \HH_{n,k}}
t_{\bf x}(H,f,g_H)$, which denote (respectively) the formal weighted
power series and the function that it defines.

We now define a family of distinguished subsets of $\WW$.

\begin{definition}\label{DWsubsets}
Given $I \subset \R$, define $\WW_I$ to be the set of all $f \in \WW$
with image in $I$, and
\begin{equation}
\Wlimit{r} := \bigcup_{0 \leq s < r} \Wopen{s} = \bigcup_{0 \leq s < r}
\Wclosed{s}, \qquad 0 < r \leq \infty.
\end{equation}
\end{definition}

We have the useful observation that if $I = [a,b]$ with $|a| \leq b$,
then a weighted $(k,{\bf x}$)-labelled power series $F$ is absolutely
convergent on $\WW_I$ if and only if $F$ is absolutely convergent on
$\WW_{[-b,b]}$. Moreover, $F$ is absolutely convergent on $\Wlimit{r}$ if
and only if $F$ is absolutely convergent on the constant graphons
$(0,r)$.

\begin{remark}\label{Rconv}
Consider the special case where $k = 0$, $g_H \equiv a_H$ is constant for
all $H$, and $a_H = 0$ if $H \notin \G$. It is clear by the theory of
rearrangements of series that if $\sum_n s^n \sum_{H \in \G_n} |a_H|$
converges then so does $\sum_{H \in \G} a_H t(H,f)$ for all $f \in \WW$
with image in $[-s,s]$. On the other hand, convergence at all constant
graphons need not guarantee convergence on all of $\WW$. For instance,
suppose $a_{K_3^{\coprod n}} = 2^{-n} = - a_{K_{1,3n}}$ for the triangle
$K_3$ and all (bipartite) star graphs $K_{1,3n}$ with $n \geq 1$, and all
other $a_H$ are zero. Then the corresponding power series is given by:
\begin{equation}\label{Eserieseg}
\sum_{n \in \N} 2^{-n} (t(K_3,f)^n - t(K_{1,3n},f)).
\end{equation}

\noindent It is clear that the series \eqref{Eserieseg} converges at all
$f \equiv s \in \R$. However, note that $t(K_3,K_2) = 0 < 2^{-3n} =
t(K_{1,3n},K_2)$ since $K_{1,3n}$ and $K_2 = K_{1,1}$ are bipartite while
$K_3$ is not. Hence the series \eqref{Eserieseg} diverges to $-\infty$ at
$s f^{K_2}$, for all $s \geq 2^{4/3}$.
\end{remark}

In light of Remark \ref{Rconv}, we introduce the following notation.

\begin{definition}\hfill
\begin{enumerate}
\item Given ${\bf x} \in [0,1]^k$ and $g_H$ a bounded measurable function
on $[0,1]^{|V(H)|}$, define the semi-norm
\begin{equation}
\|g_H\|_{1,{\bf x}} := \int_{[0,1]^{V_0}} |g_H| \prod_{i \in V_0} dx_i =
t_{\bf x}(H,1,|g_H|).
\end{equation}

\item Given a weighted power series $F(-) := \sum_{n \geq 0} \sum_{H \in
\HH_{n,k}} t_{\bf x}(H,-,g_H)$, define its {\it radius of convergence}
via:
\begin{equation}\label{Eradius}
R_F^{-1} := \limsup_{n \to \infty} \left( \sum_{H \in \HH_{n,k}}
\|g_H\|_{1,{\bf x}} \right)^{1/n}.
\end{equation}
\end{enumerate}
\end{definition}

We now establish several fundamental properties of weighted $(k,{\bf
x})$-labelled power series. The proofs combine standard analysis
arguments from \cite[Chapter 7]{Rudin:1964} while keeping track of the
combinatorics of multigraphs that arises in the present setting.

Our first result justifies the use of the name ``radius of convergence".

\begin{prop}\label{Pradius}
Every weighted power series $F(-) = \sum_n \sum_{H \in \HH_{n,k}} t_{\bf
x}(H,-,g_H)$ converges absolutely on $\Wlimit{R_F}$. Moreover, this
convergence is uniform on $\Wclosed{s}$ for any $0 \leq s < R_F$. For all
$r \in (R_F, \infty)$, there exist functions $f \in \Wlimit{r}$ such that
$F(f)$ is not absolutely convergent.
\end{prop}

Our second result shows that convergent weighted $(k, {\bf x})$-labelled
power series are closed under addition, multiplication, and
differentiation.

\begin{theorem}\label{thm:weight_series}
Suppose $c,d \in \R$, 
\[ F(-) = \sum_{n \geq 0} \sum_{H \in \HH_{n,k}}  t_{\bf x}(H,-,g_H),
\qquad G(-) = \sum_{n \geq 0} \sum_{H \in \HH_{n,k}}  t_{\bf x}(H,-,g'_H)
\]

\noindent are two weighted $(k, {\bf x})$-labelled power series, and $R
:= \min(R_F, R_G)$ is positive. Then the following three properties hold.
\begin{enumerate}[(i)]
\item For all $f \in \Wlimit{R}$,
\begin{equation}
(c F + d G)(f) = \ \sum_{n \geq 0} \sum_{H \in \HH_{n,k}} t_{\bf
x}(H,f,cg_H + dg'_H)
\end{equation}
where the right-hand side is a weighted $(k, {\bf x})$-labelled power
series with
$R_{cF + dG}  \ge R$.

\item For all $f \in \Wlimit{R}$,
\begin{equation}
(FG)(f) = \ \sum_{n \geq 0} \sum_{H \in \HH_{n,k}}  t_{\bf x}( H, f,
\sum_{H_1 H_2 = H} g_{H_1} g'_{H_2})
\end{equation}
where the right-hand side is a weighted $(k, {\bf x})$-labelled power
series with $R_{FG}  \ge R$.

\item Moreover, $F$ is G\^ateaux-smooth on $\Wlimit{R_F}$. More
precisely, for all $f_0 \in \Wlimit{R_F}$,
\begin{equation}
d^m F(f_0; f_1, \dots, f_m) = \sum_{n \geq m} \sum_{H \in \HH_{n,k}} d^m
t_{\bf x}(H,-,g_H)(f_0; f_1, \dots, f_m), \qquad \forall f_1, \dots, f_m
\in \WW,
\end{equation}

\noindent where every summand on the right-hand side is a linear
combination of terms of the form $t_{\bf x}(H',-,g_{H'})$, and together
they form a weighted $(k, {\bf x})$-labelled power series such that
$R_{d^m F(-; f_1, \dots, f_m)} \ge R_F$.
\end{enumerate}
\end{theorem}\smallskip

\noindent {\bf Application 3: Right homomorphism densities.}

As an application of this work on series, we address a question posed by
Lov\'asz in his list of open problems. Namely, Lov\'asz asks in
\cite[Problem 16]{Lovasz:Open} if there is a way to find a formula for
right homomorphism densities $t(-,G)$ in terms of left homomorphism
densities $t(H,-)$.
A problem is that $t(-,G)$ is not continuous in the cut-norm.\footnote{In
fact, $t(-,G)$ is not even well defined. For example $n^2t(H,c) =
t(H[n],c)$ where $H[n]$ is the $n$-fold blowup of $H$, but $f^H =
f^{H[n]}$ for all $n$.}
The theory of right-convergence proposes several closely related natural
remedies (see \cite[Chapter 12]{Lovasz:2013}). The only proposal that is
a continuous function on graphons is the {\it overlay functional}
\begin{equation*}
C(U,W) := \sup_{\phi \in S_{[0,1]}} \int_{[0,1]^2} U(x,y)
W(\phi(x),\phi(y)) dx dy, \qquad U,W \in \WW.
\end{equation*}

\noindent (See \cite[Lemma 12.7]{Lovasz:2013}.) We now show that $C(U,W)$
cannot be expanded as an absolutely convergent sequence in general.  

\begin{prop}
There exists $W \in \W$ for which there is no sequence of constants $a_H$
such that the series $\sum_{H \in \HH} a_H t(H,-)$ is absolutely
convergent and equals $C(-,W)$ on $\W$.
\end{prop}

\begin{proof}
From the definitions, note that $dC(-,W)(0;U) = C(U,W)$ for all $U,W \in
\W$. Since $C(U,W)$ is not bilinear, there exists $W \in \W$ such that
$dC(-,W)(0;U)$ is not linear in $U$.  
Therefore it cannot be $C^1$ at 0.  If $C(-,W)$ were expandable in terms
of an absolutely convergent series of multigraph homomorphism densities
$t(H,f)$, then it would be G\^ateaux smooth by Theorem
\ref{thm:weight_series}, so such an expansion is impossible.
\end{proof}

\subsection{Taylor series}\label{Staylor}
The Stone-Weierstrass type Theorem \ref{Stone} implies that the linear
span $\scrs_t$ of homomorphism densities $t(H,-)$ is a dense subalgebra
of $C(\W / \sim, \R)$. In this (non-constructive) sense, homomorphism
densities can be used to approximate continuous class functions on $\W$.
The goal of this subsection is to show how to write Taylor expansions for
smooth functions on $\W / \sim$ around $0$ in terms of homomorphism
densities. In addition, we provide sufficient conditions for the
convergence of the Taylor series of a function $F$ to the function $F$.
The main idea is that the argument for Theorem \ref{thm:main} shows that
the multilinear derivatives of smooth class functions can be represented
by those of $t(H,f)$. We now apply the above ideas to define Taylor
polynomials for smooth class functions and prove a Taylor's theorem for
them.

\begin{theorem}\label{thm:taylor}
Suppose $F : \W \to \R$ is a $C^n$ class function. Then:
\begin{enumerate}[(i)]
\item $F$ has a Taylor polynomial $P_n(f) := \sum_{m=0}^n
\frac{1}{m!} d^m F(0; f, f, \dots, f)$. The remainder is:
\begin{equation*}
R_n(f) := F(f) - P_n(f) = \frac{d^{n+1}F(c_f f;f, \dots, f)}{(n+1)!}
\end{equation*}

\noindent for some $c_f \in [0,1]$.

\item From $F$ one can uniquely define scalars $a_H$ for all $H \in
\HH_{\leq n}$, such that for all $f \in \bigcup_{\sigma \in
S_{[0,1]}}\sigma(\WW_{\bf p})$,
\begin{equation}\label{Emthder}
\frac{1}{m!} d^m F(0; f, f, \dots, f) = \sum_{H \in \HH_m} a_H t(H,f)
\end{equation}

\noindent for all integers $0 \leq m \leq n$.
If in addition the higher derivatives $\frac{1}{m!} d^m F(0; f, f,
\dots, f)$ are continuous in the $L^1$ topology, then
Equation \eqref{Emthder} holds for all $f \in \W$ and all integers $0
\leq m \leq n$.
\comment{
\[\frac{1}{m!} d^m F(0; f, f, \dots, f) = \sum_{H \in \HH_m} a_H t(H,f)
\]

\noindent for all integers $0 \leq m \leq n$ and all $f \in \W$.
}
\end{enumerate}
\end{theorem}

\begin{proof}\hfill
\begin{enumerate}[(i)]
\item Define the function $\mathcal{F}(t) := F(tf)$ for a fixed direction
$f \in \W$. Then $\mathcal{F}$ is $C^n$ and this result follows from the
one variable Taylor's theorem for $\mathcal{F}$.

\item Since $F$ is a $C^n$ class function, we know by Proposition
\ref{prop:Cn_consistent} that for $0 \leq m \leq n$, $\frac{1}{m!} d^m
F(0; f_1, f_2, \dots, f_m)$ is a symmetric $S_{[0,1]}$-invariant
multilinear function.
Now using Theorem \ref{thm:hom_prop}(iv), we can determine a unique set
of coefficients $a_H$ such that Equation \eqref{Emthder} holds for all $f
\in \bigcup_{\sigma \in S_{[0,1]}}\sigma(\WW_{\bf p})$.
Now suppose the higher derivatives are continuous in the $L^1$ topology.
Recall by Proposition \ref{prop:cont_func_char}(i) that $\sum_{H \in
\HH_m} a_H t(H,f)$ is also continuous in the $L^1$ topology on $\W$. The
result now follows by the density of $\WW_{\bf p}$ in $\W$.
\end{enumerate}
\end{proof}

Thus, given a smooth class function $F : \W \to \R$, we can use Theorem
\ref{thm:taylor} to define an infinite Taylor series
\begin{equation}\label{Etaylor}
P(F)(-) := \sum_{m=0}^\infty \sum_{H \in \HH_m} a_H t(H,-),
\end{equation}

\noindent where
\[ \sum_{H \in \HH_m} a_H t(H,f) = \frac{1}{m!} d^m F(0; f, f, \dots, f),
\qquad \forall m \geq 0,\ f \in \bigcup_{\sigma \in S_{[0,1]}}
\sigma(\WW_{\bf p}). \]

\noindent Note that $P(F)$ is in fact a weighted power series as in
Definition \ref{Dseries} (with $m=0$ and $g_H \equiv a_H$ constant for
all $H \in \HH$).

Given a smooth function $F : \W \to \R$, a natural question to ask is if
the Taylor series defined above converges to $F$. Recall from the
one-variable Taylor theory that there exist nonzero smooth functions $F$
on $\R$, all of whose derivatives vanish at the origin (and so $F$ has
trivial Taylor polynomials). We now show that a similar phenomenon occurs
for graphons.
Namely, consider the function $F(f) := e^{-1/t(H,f)}$ for a finite simple
graph $H$ with at least one edge (with $F(f) := 0$ if $t(H,f) = 0$).
Then all higher G\^ateaux derivatives of $F$ vanish at the origin.
Indeed, fix a direction $g \in \W$ and set $\mathcal{F}(c) :=
e^{-1/t(H,cg)} = e^{-1/c^{|E(H)|}t(H,g)}$.  Now the G\^ateaux derivatives
of $e^{-1/t(H,f)}$ at $0$ in the $g$ direction are just one-sided
derivatives of $\mathcal{F}(c)$. Moreover, $\mathcal{F}(c) = e^{-A/c^n}$,
where $A = 1/t(H,g)$ and $n = |E(H)|$.  Taking the derivative yields
$\mathcal{F}'(c) = e^{-A/c^n} \frac{n}{c^{n+1}}$. Higher derivatives are
all of the form $\mathcal{F}^{(m)}(c) = e^{-A/c^n} R_m(c)$ where $R_m$ is
a rational function.  Moreover, in all cases $\mathcal{F}^{(m)}(0) = 0$
for all $m$. However, $\mathcal{F}$ is not the zero function.\medskip

We now provide a sufficient condition under which the Taylor series of a
smooth function $F$ converges to $F$.

\begin{theorem}\label{thm:analytic}
Suppose $F : \WW \to \R$ satisfies the following assumptions:
\begin{enumerate}
\item $F$ is G\^ateaux smooth, continuous in the cut-norm, and a class
function.

\item For all $\{ 0, 1 \}$-valued graphons $f \in \W$, the Taylor
polynomials
\begin{equation*}
P_n(f) := \sum_{m=0}^n \frac{1}{m!} d^m F(0; f, f, \dots, f)
\end{equation*}
converge to $F(f)$ as $n \to \infty$.

\item The power series $P(F)$ given by \eqref{Etaylor} is absolutely
convergent on $\W$.
\end{enumerate}

\noindent Then $a_H = 0$ for all $H \in \HH \setminus \G$, and
\[ F(f) = P(F)(f) = \sum_{m \geq 0} \sum_{H \in \G_m} a_H t(H,f), \qquad
\forall f \in \W. \]
\end{theorem}

\noindent In other words, the Taylor series of $F$ converges to $F$ on
all of $\W$.

\begin{proof}
Define the weighted power series
$\ptild(F)(-) := \sum_{m \geq 0} \sum_{H \in \HH_m} a_H
t(H^{simp},-)$,
where $H^{simp}$ is the simple graph obtained from $H$ by replacing each
set of repeated edges between a pair of vertices by one edge. Then
$\ptild(F)$ is also absolutely convergent on $\W$. Indeed, this is
equivalent to
$\sum_{m \geq 0} \sum_{H \in \HH_m} |a_H| t(H^{simp},1)$
converging and that follows from rearranging, by absolute
convergence, the terms of the convergent sequence
$\sum_{m \geq 0} \sum_{H \in \G_m} |a_H| t(H,1)$.

\noindent In addition, $\ptild(F)$ is continuous in the cut-norm because
it is a uniform limit of continuous functions on $\W$. Moreover,
$\ptild(F)(f) = P(F)(f)$ for all $\{ 0, 1 \}$-valued graphons $f \in \W$
since $t(H^{simp},f) = t(H,f)$ for such $f$.  By assumption, $F(f) =
P(F)(f)$ for $f$ a $\{ 0, 1 \}$-valued graphon, so $F(f) = \ptild(F)(f)$
for such $f$.  Since both $F$ and $\ptild(F)$ are continuous on $\W$ and
equal on all finite simple graphs, we conclude that they are equal on all
$f \in \W$.

Now by Theorem \ref{thm:taylor}(ii), $P(F)$ and $\ptild(F)$ have the same
G\^ateaux derivatives at $0$ along the directions in $\WW_{\bf p}$. We
will show in the next subsection (Theorem \ref{thm:gen_whitney}) that a
power series with a positive radius of convergence is uniquely determined
by its G\^ateaux derivatives at $0$ along the directions in $\WW_{\bf
p}$. We conclude that $\ptild(F) = P(F)$ as weighted power series.
\end{proof}

\subsection{Uniqueness of Taylor series and linear independence}\label{Suniqueness}

In the previous section, we showed how to represent the Taylor series of
a smooth class function $F$ around $0$, in terms of homomorphism
densities. We also provided sufficient conditions for when such a Taylor
series expansion is absolutely convergent, and converges to $F$ on all of
$\W$. We left open the question of whether or not this expansion is
unique.  The next theorem shows that such an expansion is indeed unique.
In fact, we prove this is true for arbitrary $(k,{\bf x})$-labelled power
series when the entries of ${\bf x}$ are distinct. The main ingredient is
a generalization of the crucial linear independence result of Theorem
\ref{thm:hom_prop} to functions $t_{\bf x}(H,-)$.

\begin{theorem}\label{thm:gen_whitney}
Fix an integer $k \geq 0$ and a vector ${\bf x} \in [0,1]^k$ with
distinct entries. Suppose the formal series $\sum_{n=0}^\infty \sum_{H
\in \HH_{n,k}} t_{\bf x}(H,-,a_H)$ has a positive radius of convergence.
Then the coefficients $a_H$ for $\bigcup_{n=0}^\infty \HH_{n,k}$ are
uniquely determined by the function $F(f) := \sum_{n=0}^\infty \sum_{H
\in \HH_{n,k}} t_{\bf x}(H,f,a_H)$.  More precisely, the coefficients
$a_H$ can be recovered from the derivatives $d^nF(0;f, \dots, f)$ for $f
\in \WW_{\bf p}$ and $n \ge 0$.
\end{theorem}

\begin{proof}[Proof of Theorem \ref{thm:gen_whitney}]
By part (iii) of Theorem \ref{thm:weight_series} and Lemma
\ref{lem:multilinear_deriv}, note that
\[ \frac{1}{n!} (d^n F)(0; f, f, \dots, f) = \sum_{H \in \HH_{n,k}}
t_{\bf x}(H,f, a_H), \qquad \forall n \geq 0. \]

\noindent Therefore, the theorem follows from the following
generalization of Corollary \ref{cor:whitney}:\smallskip

{\it Let $n,k \geq 0$ be fixed integers, and $\bf{x}$ be a vector of $k$
distinct constants ${\bf x}_1, \dots, {\bf x}_k \in [0,1]$. Then the
functions $t_{\bf x} (H,-,1)$ are linearly independent for $H \in
\HH_{n,k}$.}\smallskip

To prove this claim, note that if the functions $t_{\bf x}(H_i,-) =
t_{\bf x}(H_i, -, 1)$ are linearly dependent, then so are $d^m t_{\bf
x}(H_i,-;g_1, \dots, g_m)$ for all $m$ and all choices of tuples ${\bf g}
= (g_1, \dots, g_m)$. Thus we will set $m=n$ and produce tuples $\{ {\bf
g}(G) : G \in \HH_{n,k} \}$ such that the matrix
\[ M := (( d^n t_{\bf x}(H,0; {\bf g}(G) ))_{H,G \in \HH_{n,k}} \]

\noindent is nonsingular. Indeed, recall that in the $k =0$ case, we
showed linear independence, in Corollary \ref{cor:whitney}, by picking the
${\bf g}(G)$ such that $d^n t(H,0; {\bf g}(G)) = T(t(H,-))_n(G,p)$. In
that case, the matrix $M$ was triangular.

We now generalize this argument to $\HH_{n,k}$. To that end, fix an
integer
\[ p > 2n+\frac{2}{\min_{1 \leq i < j \leq k} |{\bf x}_i - {\bf x}_j|} \]

\noindent with $p$ relatively prime to the denominators of any ${\bf
x}_i$ that are rational. Then $p$ is finite because the ${\bf x}_i$ are
distinct. We will now study the G\^ateaux derivatives of $t_{\bf x}(H,-)$
along the directions
\[ f = e^p_{(a,b)} = {\bf 1}_{\left( \frac{a-1}{p}, \frac{a}{p} \right]
\times \left( \frac{b-1}{p}, \frac{b}{p} \right]} + {\bf 1}_{\left(
\frac{b-1}{p}, \frac{b}{p} \right] \times \left( \frac{a-1}{p},
\frac{a}{p} \right]}. \]

\noindent By choice of $p$, the ${\bf x}_i$ lie in the interiors of
distinct intervals of the form $[a_i/p, (a_i+1)/p]$.

Now for every fixed $G \in \HH_{n,k}$, fix an injective map $\phi :
V(G) \to \{1, \dots, p\}$ such that for each labelled vertex $v \in
V(G)$, $\phi(v) := \lceil p {\bf x}_{l_G^{-1}(v)} \rceil$. This is
possible by choice of $p$. We now define the tuples ${\bf g}(G)$ by:
${\bf g}(G)_e := e^p_{(\phi(e_s), \phi(e_t))}$ for $e \in E(G)$. Then for
all $H \in \HH_{n,k}$,
\begin{align}\label{EMHG}
M(H,G) & := d^n t_{\bf x}(H,0;({\bf g}(G)_e)_{e \in E(G)}) \notag\\
&= \sum_{\sigma: E(H) \twoheadrightarrow E(G)} \int_{[0,1]^{V_0}}
\prod_{e \in E(H)} {\bf g}(G)_{\sigma(e)}(x_{e_s}, x_{e_t}) \prod_{i \in
V_0} dx_i,
\end{align}

\noindent where the last equality follows by Lemma
\ref{lem:multilinear_deriv}. (Note that the order of G\^ateaux
differentiation does not matter since mixed partials are equal.)

We now claim that $M(H,G) = |\Surj(H,G)|/p^{|V_0|}$, where $\Surj(H,G)$
is the set of node-and-edge maps from $H$ to $G$ (in the sense of
Definition \ref{Dgraph-map}) that are surjective. To prove the claim, let
$\tau' : V(H) \setminus V_0 \to \{1, \dots, p\}$ be defined for labelled
vertices $v \in V(H)$ by: $\tau'(v) := \lceil p x_v \rceil$. Consider an
arbitrary term in \eqref{EMHG}.  Then ${\bf g}(G)_e$ is constant on each
``sub-rectangle" in $[0,1]^{V_0}$ of size $1/p^{|V_0|}$, hence
\begin{align}\label{EMHG2}
&\ \int_{[0,1]^{V_0}} \prod_{e \in E(H)} {\bf g}(G)_{\sigma(e)}(x_{e_s},
x_{e_t}) \prod_{i \in V_0} dx_i \notag\\
= &\ \frac{1}{p^{|V_0|}} \sum_{\substack{\tau: V(H) \to \{1, \dots, p\}
\\ \tau \text{ extends } \tau'}} \quad \prod_{e \in E(H)}
{\bf g}(G)_{\sigma(e)}\left(\frac{\tau(e_s)-0.5}{p}, \frac{\tau(e_t)-
0.5}{p}\right).
\end{align}

By our choice of ${\bf g}(G)_e$ we have
\begin{equation*}
{\bf g}(G)_{\sigma(e)}\left(\frac{\tau(e_s)- 0.5}{p}, \frac{\tau(e_t)-
0.5}{p}\right) =
\begin{cases}
1 & \text{ if } \{\tau(e_s), \tau(e_t)\} = \{ \phi(\sigma(e)_s),
\phi(\sigma(e)_t)\},\\
0 &\ \text{otherwise}.
\end{cases}
\end{equation*}

\noindent Note that the product of the above expression (over $e \in
E(H)$) is zero unless $\tau(V(H)) \subset \phi(V(G))$.
Therefore, the right side of Equation \eqref{EMHG2} can be written as
\begin{equation*}
\frac{1}{p^{|V_0|}}\sum_{\substack{\tau: V(H) \to \phi(V(G))\\ \tau
\text{ extends } \tau'}} \quad \prod_{e \in E(H)} {\bf
g}(G)_{\sigma(e)}\left(\frac{\tau(e_s)- 0.5}{p}, \frac{\tau(e_t)-
0.5}{p}\right).
\end{equation*}

\noindent Since $\phi$ is injective, we can define vertex maps
$\phi^{-1}\tau: V(H) \to V(G)$ for every such $\tau$. Note that
$\phi^{-1} \tau$ sends labelled vertices of $H$ to the corresponding
labelled vertices in $G$, because $\phi(v) := \lceil p {\bf
x}_{l_G^{-1}(v)} \rceil$ and $\tau$ extends $\tau'$.
We recognize the sum on the right hand side of Equation \eqref{EMHG2} to be
equal to $\frac{1}{p^{|V_0|}}$ times the number of vertex maps
$\phi^{-1}\tau: V(H) \to V(G)$ that form a map of multigraphs $H \to G$,
when combined with the edge map $\sigma: E(H) \twoheadrightarrow E(G)$.  
In addition, Equation \eqref{EMHG} sums over all surjective maps $\sigma:
E(H) \twoheadrightarrow E(G)$ so $M(H,G)$ $= |\Surj(H,G)|/p^{|V_0|}$ as
claimed.

Now note that $(M(H,G))_{H,G \in \HH_{n,k}}$ is triangular with nonzero
diagonal entries, when $\HH_{n,k}$ is partially ordered consistent with
the existence of surjections. Hence $(M(H,G))_{H,G \in \HH_{n,k}}$ is an
invertible matrix, which concludes the proof.
\end{proof}

\begin{remark}
We now explain more generally why several of the integral formulas we
have examined above, can be interpreted as combinatorial quantities.
Let $H$ be a $k$-labelled multigraph with unlabelled vertices $V_0$, and
say ${\bf x} \in [0,1]^k$ are fixed irrational numbers, such that $x_v =
{\bf x}_{l_H^{-1}(v)}$, as above. For fixed $p \in \N$, define $\tau'(v)
= \lceil p {\bf x}_{l_H^{-1}(v)} \rceil$ for labelled vertices $v$. Now
consider the expression
\begin{equation}\label{Ecomboformula}
\int_{[0,1]^{V_0}} \prod_{e \in E(H)} g_{e}(x_{e_s}, x_{e_t}) \prod_{i
\in V_0} dx_i
= \frac{1}{p^{|V_0|}}\sum_{\substack{\tau: V(H) \to \{1, \dots, p\},\\
\tau \text{ extends } \tau'}} \quad \prod_{e \in E(H)}
g_{e}\left(\frac{\tau(e_s)- 0.5}{p}, \frac{\tau(e_t)- 0.5}{p}\right)
\end{equation}

\noindent where $g_e = f^{G_{e}}$, and $G_e$ are simple graphs on the
vertex set $\{1, \dots, p\}$ for each edge $e \in E(H)$.

Just as in the proof of Theorem \ref{thm:gen_whitney}, 
\begin{equation*}
\prod_{e \in E(H)} g_{e}\left(\frac{\tau(e_s)-0.5}{p}, \frac{\tau(e_t)-
0.5}{p}\right) =
\begin{cases}
1 &\ \text{ if } \{\tau(e_s), \tau(e_t)\} \in E(G_e)\ \forall e \in
E(H),\\
0 &\ \text{ otherwise}.
\end{cases}
\end{equation*}

\noindent Therefore the quantity in Equation \eqref{Ecomboformula} is
equal to $\frac{1}{p^|V_0|}$ times the number of maps $\tau: V(H) \to
\{1, \dots, p\}$ with $\tau$ extending $\tau'$ such that $\{\tau(e_s),
\tau(e_t)\} \in E(G_e)$ for each $e \in E(H)$.

In the special case $G_e = G$ for a single graph $G$ (for all $e$), the
above analysis shows that $t(H,f^G) = \hom(H,G) / |V(G)|^{|V(H)|}$. A
similar formula can be obtained by extending this analysis to multigraphs
$G_e$, by weighting the edges of $G_e$ according to the multiplicity of
edges found.

Sometimes it is useful to consider multiple sets of graphs $\{ G_e : e
\in E(H) \}$ simultaneously. In particular, one can combinatorially
interpret the derivatives
\begin{align*}
& d^m t_{\bf x}(H,f^G;g_1,g_2, \dots, g_m) = \\
& \sum_{\substack{A \subset E(H)\\ |A| = m}} \ \sum_{\sigma : \{1, \dots,
m\}\twoheadrightarrow A} \int_{[0,1]^{V_0}} \prod_{l=1}^m
g_{l}(x_{\sigma(l)_s}, x_{\sigma(l)_t}) \prod_{e \in E(H) \setminus A}
f^G(x_{e_s},x_{e_t}) \prod_{i \in V_0} dx_i, \ \forall f^G, g_i \in
\WW_{\bf p}
\end{align*}

\noindent for $0 \leq m \leq n$, $0 \leq k$, $H \in \HH_{n,k}$. In this
paper we have specialized to the case $m = n$, for the proof of Theorem
\ref{thm:gen_whitney}. In that proof we picked the $G_e$ to each be a
single edge, and let the $G_e$ vary over all edges $e \in E(H)$ to form a
multigraph $G$. Since $G_e$ was a single edge, it forced $\{\tau(e_s),
\tau(e_t)\} = G_e$. The sum over all $\sigma$ in the above equation
allowed us to interpret this derivative in terms of the number of
surjective maps $: H \twoheadrightarrow G$. We were thus able to obtain
linear independence results about homomorphism densities in an analytic
way.
\end{remark}

\subsection{Infinite quantum algebras}\label{Sinfq}

We now explore a question raised by Lov\'asz regarding the algebras
$\Q_k$. Lov\'asz asks in \cite[Problem 7]{Lovasz:Open} if it is possible
to extend the definition of $\Q_k$ to infinite sums of $k$-labelled
multigraphs. One answer to this question is to interpret $\Q_k$ as the
graded vector space $\bigoplus_{n \geq 0} \R^{\HH_{n,k}}$; then an
extension to infinite sums would simply be the larger space $\qhat_k :=
\prod_{n \geq 0} \R^{\HH_{n,k}}$.

Let $\qhat_{k,{\bf x}}$ denote the set of all weighted $(k,{\bf
x})$-labelled power series with constant coefficients $g_H \equiv a_H$
(for all $H$). Then note that $\qhat_{k,{\bf x}}$ is a unital
$\R$-subalgebra of the commutative algebra studied in Proposition
\ref{Palgebra}.
Moreover, the obvious map $: \qhat_k \to \qhat_{k,{\bf x}}$ sending the
tuple $\{ a_H : H \in \bigcup_n \HH_{n,k} \}$ to $\sum_{n \geq 0} \sum_{H
\in \HH_{n,k}} t_{\bf x}(H,-,a_H)$ is a vector space isomorphism.
Therefore $\qhat_k$ inherits an algebra structure from $\qhat_{k,{\bf
x}}$, which we note is independent of ${\bf x}$ and extends the algebra
structure on $\Q_k$.

\begin{remark}
Note that $\qhat_k$ can also be interpreted as the completion of the
topological graded vector space $\Q_k$ under the metrizable topology --
in fact, the translation-invariant metric -- defined by the grading. In
this topology, the subspaces $\bigoplus_{n \geq N} {\rm span}_\R
\HH_{n,k}$ -- which are in fact graded ideals -- form a fundamental
system of neighborhoods of $0$. Moreover, it is not hard to show that the
algebra operations on $\qhat_k$ (defined by Proposition \ref{Palgebra})
are continuous in this topology. In this sense the grading introduced in
this paper provides an algebraic candidate to Lov\'asz's question. This
candidate is essentially unique by the universality of completions.
\end{remark}

\begin{remark}\label{Rgraphalgebra}
One can show that the span $\Q_\N$ of $\HH'$ is a polynomial algebra
$\R[{\bf X}]$, whose set of generators ${\bf X}$ consists of all
(isomorphism classes of) connected multigraphs $H \in \HH'$ of one of two
kinds:
(1) $H$ contains exactly two vertices, both of which are labelled and
adjacent with a unique edge; or
(2) the subset of labelled vertices in $V(H)$ is independent, and
removing this subset (and all edges adjacent to it) from $H$ does not
disconnect the resulting induced sub-multigraph. Similarly for each $k
\geq 0$, $\Q_k = \R[{\bf X}_k]$ is also a polynomial algebra, with
generators ${\bf X}_k := {\bf X} \cap \Q_k$.
Note that the completion $\qhat_k$ can be identified with ``formal power
series" in ${\bf X}_k$.
\end{remark}

Having explored Lov\'asz's question algebraically, we now explore how to
apply analytical techniques to infinite series of algebras. Note that
this can be done for elements of $\Q_k$ by using the map $\alpha_{\bf
x}$. For instance if $k=0$, then by Corollary \ref{cor:whitney}, the map
$\alpha : \Q_0 \to \scrs_t$ is an algebra isomorphism that identifies
multigraphs with their homomorphism densities, which are functions on
$\W$ and hence amenable to analytical treatment.

Thus the immediate goal is to try extending the map $\alpha_{\bf x} :
\Q_k \to Func(\W,\R)$ to convergent $(k,{\bf x})$-labelled power series
in $\qhat_k$.

\begin{definition}
Given $0 < R \leq \infty$, define $\Q_{k,R}$ to be the set of formal
$(k,{\bf x})$-labelled power series of the form $\sum_{n=0}^\infty
\sum_{H \in \HH_{n,k}} t_{\bf x}(H,-,a_H)$ (for some ${\bf x} \in
[0,1]^k$) whose radius of convergence is greater than $R$. Now define
$\Q_{\N,R} := \bigcup_{k \geq 0} \Q_{k,R}$ for $0 < R \leq \infty$.
\end{definition}

\begin{remark}
Note that all values of ${\bf x} \in [0,1]^k$ yield the same set of power
series in $\Q_{k,R}$, since the radius of convergence defined in
\eqref{Eradius} does not depend on ${\bf x}$ if all $a_H$ are constant.
Also note that the $\Q_{k,R}$ constitute a two-parameter family of
commutative unital graded $\R$-algebras that is decreasing in $0 < R \leq
\infty$ and increasing in $k \geq 0$, by Theorem \ref{thm:weight_series}.
\end{remark}

We now have the following result which shows when infinite formal series
of graphs can be embedded into spaces of functions amenable to analytic
treatment.

\begin{theorem}\label{thm:alpha_injection}
Fix ${\bf x} \in [0,1]^\N$, and $0 <s < R \leq \infty$. The algebra map
$\alpha_{\bf x}: \Q_\N \to Func(\WW_{[-s,s]},\R)$ of $\R$-algebras can be
extended continuously to $\Q_{\N,R}$ (with respect to the topology of
$\bigcup_{k \geq 0} \qhat_k$, and uniform convergence in
$Func(\WW_{[-s,s]},\R)$).

Furthermore, $\alpha_{\bf x}$ is an embedding of $\R$-algebras if and
only if the $x_i$ are distinct.
\end{theorem}

\noindent Note that continuously extending $\alpha_{\bf x}$ from $\Q_\N$
to $\Q_{\N,R}$ is equivalent to continuously extending $\alpha_{\bf x}$
from $\Q_k$ to $\Q_{k,R}$ for each $k \geq 0$.

\begin{proof}
The map $\alpha_{\bf x}$ can be extended by Proposition \ref{Pradius}
from $\Q_k$ to $\Q_{k,R}$.  The extension is continuous because
Proposition \ref{Pradius} guarantees uniform convergence of the series.
Now the first two parts of Theorem \ref{thm:weight_series} show that
$\alpha_{\bf x} : \Q_{k,R} \to Func(\WW_{[-s,s]},\R)$ is an algebra map
for each $k \geq 0$. The result for $\Q_{\N,R}$ follows by compatibility
across $k \geq 0$.

By Theorem \ref{thm:gen_whitney}, $\alpha_{\bf x}$ is injective when the
$x_i$ are distinct. Now if $H_0 \in \Q_{\N,R}$ is in $\ker \alpha_{\bf
x}$, then $H_0 \in \Q_{k,R}$ for some $k$, whence $H_0 = 0$.

Finally, assume that for ${\bf x} \in [0,1]^k$, there exist two $x_i$
that are equal.  Without loss of generality, we can assume that $x_1 =
x_2$. Consider now any $k$-labelled graph $H$ and the graph $H'$ that
swaps the vertex labelled 1 with the vertex labelled 2.  Then
$t_{\bf x}(H,f) = t_{\bf x}(H',f)$
for all $f \in \W$ so $\alpha_{\bf x}$ is not injective on $\Q_{k,R}$.
The result for $\Q_{\N,R}$ follows immediately.
\end{proof}

Note that the images of the maps $\alpha_{\bf x}$ are inter-related for
different ${\bf x}$ as follows. Given $k \in \N$ and ${\bf x} \in
[0,1]^k$, and $H \in \HH_{n,k}$,
\begin{equation*}
t_{\bf x}(H,f) =  t_{\sigma({\bf x})}(H,f^{\sigma^{-1}}) =:
\sigma(t_{\sigma({\bf x})}(H,-))(f), \qquad \forall \sigma \in S_{[0,1]},
f \in \WW.
\end{equation*}

\noindent Now if ${\bf x}$ and ${\bf y}$ both have pairwise distinct
elements, then for any $\sigma \in S_{[0,1]}$ such that $\sigma({\bf x})
= {\bf y}$, we get that $\sigma(\img \alpha_{\bf x}) = \img \alpha_{\bf
y}$. Thus, the image of the maps $\alpha_{\bf x}$ are the same up to the
action of $S_{[0,1]}$ on $Func(\WW_{[-s,s]}, \R)$, with $s$ as in Theorem
\ref{thm:alpha_injection}.\medskip

\noindent {\bf Concluding remarks.}

It is now possible to concretely explain the analogy in Section
\ref{Sprelim} between homomorphism densities and monomials, with degree
the number of edges. Namely, using Remark \ref{Rgraphalgebra}, it is
clear that the set of (unlabelled) multigraphs $\HH$ spans the polynomial
algebra $\Q_0 = \R[{\bf X}_0]$ -- and hence serves as a family of
monomials in the generators ${\bf X}_0$, with degree given by the number
of edges. 
Now Theorem \ref{thm:alpha_injection} provides a canonical (up to the
$S_{[0,1]}$ action) way of embedding a subalgebra of infinite formal
series of $k$-labelled graphs into $Func(\WW_{[-s,s]}, \R)$, with $s$ as
in Theorem \ref{thm:alpha_injection}. The homomorphism densities $t(H,-)$
are simply the images of the monomials in ${\bf X}_0$, under the algebra
embedding $\alpha$.  

Additionally, we have found that our notion of degree in this polynomial
algebra interacts well with G\^ateaux differentiation.  Theorem
\ref{thm:main} shows us that degree $N$ homomorphism densities are
precisely the continuous class functions that vanish after taking $N + 1$
derivatives. All of this suggests that the $\alpha$ map is a good
starting point for further investigation into the analytic theory of
infinite quantum algebras.

\subsection*{Acknowledgements}
We would like to thank Professor Amir Dembo for valuable discussions.  We
would like to thank the anonymous referees for their useful comments and
suggestions that improved the paper.



\vspace*{-2.15mm}

\begin{thebibliography}{10}

\bibitem{Aldous:ICM:2010}
David~J. Aldous.
\newblock Exchangeability and continuum limits of discrete random structures.
\newblock In {\em Proceedings of the {I}nternational {C}ongress of
  {M}athematicians. {V}olume {I}}, pages 141--153, New Delhi, 2010. Hindustan
  Book Agency.

\bibitem{AT:2010}
Tim Austin and Terence Tao.
\newblock Testability and repair of hereditary hypergraph properties.
\newblock {\em Random Structures Algorithms}, 36(4):373--463, 2010.

\bibitem{BB:2012}
Alexander Brudnyi and Yuri Brudnyi.
\newblock {\em Methods of Geometric Analysis in Extension and Trace Problems. Volume 1}, volume~102 of {\em Monographs in Mathematics}.
\newblock Birkh\"auser/Springer Basel AG, Basel, 2012.		

\bibitem{BC:2009}
Peter~J. Bickel and Aiyou Chen.
\newblock A nonparametric view of network models and {N}ewman-{G}irvan and
  other modularities.
\newblock {\em Proceedings of the National Academy of Sciences},
  106(50):21068--21073, 2009.

\bibitem{BCL:2010}
Christian Borgs, Jennifer~T. Chayes, and L{\'a}szl{\'o} Lov{\'a}sz.
\newblock Moments of two-variable functions and the uniqueness of graph limits.
\newblock {\em Geom. Funct. Anal.}, 19(6):1597--1619, 2010.

\bibitem{BCLSV:2008}
Christian Borgs, Jennifer~T. Chayes, L{\'a}szl{\'o} Lov{\'a}sz, Vera~T.
  S{\'o}s, and Katalin Vesztergombi.
\newblock Convergent sequences of dense graphs. {I}. {S}ubgraph frequencies,
  metric properties and testing.
\newblock {\em Adv. Math.}, 219(6):1801--1851, 2008.

\bibitem{BCLSV:2012}
Christian Borgs, Jennifer~T. Chayes, L{\'a}szl{\'o} Lov{\'a}sz, Vera~T.
  S{\'o}s, and Katalin Vesztergombi.
\newblock Convergent sequences of dense graphs {II}. {M}ultiway cuts and
  statistical physics.
\newblock {\em Ann. of Math. (2)}, 176(1):151--219, 2012.

\bibitem{CD:2011}
Sourav Chatterjee and Persi Diaconis.
\newblock Estimating and understanding exponential random graph models.
\newblock {\em Ann. Statist.}, 41(5):2428--2461, 2013.

\bibitem{CV:2011}
Sourav Chatterjee and S.R.S. Varadhan.
\newblock The large deviation principle for the {E}rd{\H o}s-{R}\'enyi random
  graph.
\newblock {\em European J. Combin.}, 32(7):1000--1017, 2011.

\bibitem{CV:2012}
Sourav Chatterjee and S.R.S. Varadhan.
\newblock Large deviations for random matrices.
\newblock {\em Commun. Stoch. Anal.}, 6(1):1--13, 2012.

\bibitem{DHJ:2013}
Persi Diaconis, Susan Holmes, and Svante Janson.
\newblock Interval graph limits.
\newblock {\em Ann. Comb.}, 17(1):27--52, 2013.

\bibitem{DJ:2008}
Persi Diaconis and Svante Janson.
\newblock Graph limits and exchangeable random graphs.
\newblock {\em Rend. Mat. Appl. (7)}, 28(1):33--61, 2008.

\bibitem{Elek:2012}
G{\'a}bor Elek.
\newblock Samplings and observables. {I}nvariants of metric measure spaces.
\newblock {\em arXiv:1205.6936}, 2012.

\bibitem{ES:2012}
G{\'a}bor Elek and Bal{\'a}zs Szegedy.
\newblock A measure-theoretic approach to the theory of dense hypergraphs.
\newblock {\em Adv. Math.}, 231(3-4):1731--1772, 2012.

\bibitem{ELS:1979}
Paul Erd{\H{o}}s, L{\'a}szl{\'o} Lov{\'a}sz, and Joel Spencer.
\newblock Strong independence of graphcopy functions.
\newblock In {\em Graph theory and related topics ({P}roc. {C}onf., {U}niv.
  {W}aterloo, {W}aterloo, {O}nt., 1977)}, pages 165--172. Academic Press, New
  York-London, 1979.

\bibitem{Gromov:2001}
Mikhail Gromov.
\newblock {\em Metric structures for Riemannian and non-Riemannian spaces}.
\newblock Springer, 2001.

\bibitem{Hatami:2010}
Hamed Hatami.
\newblock Graph norms and {S}idorenko's conjecture.
\newblock {\em Israel J. Math.}, 175:125--150, 2010.

\bibitem{Hoover:1982}
D.N. Hoover.
\newblock Row-column exchangeability and a generalized model for probability.
\newblock In {\em Exchangeability in probability and statistics ({R}ome,
  1981)}, pages 281--291. North-Holland, Amsterdam, 1982.

\bibitem{Lang:Analysis}
Serge Lang.
\newblock {\em Real and functional analysis}, volume 142 of {\em Graduate Texts
  in Mathematics}.
\newblock Springer-Verlag, New York, third edition, 1993.

\bibitem{LOGR:2012}
James Lloyd, Peter Orbanz, Zoubin Ghahramani, and Daniel Roy.
\newblock Random function priors for exchangeable arrays with applications to
  graphs and relational data.
\newblock In P.~Bartlett, F.C.N. Pereira, C.J.C. Burges, L.~Bottou, and K.Q.
  Weinberger, editors, {\em Advances in Neural Information Processing Systems
  25}, pages 1007--1015. 2012.

\bibitem{Lovasz:Open}
L{\'a}szl{\'o} Lov{\'a}sz.
\newblock Graph homomorphisms: Open problems.
\newblock 2008.

\bibitem{Lovasz:2011}
L{\'a}szl{\'o} Lov{\'a}sz.
\newblock Subgraph densities in signed graphons and the local
  {S}imonovits-{S}idorenko conjecture.
\newblock {\em Electron. J. Combin.}, 18(1):Paper 127, 21, 2011.

\bibitem{Lovasz:2013}
L{\'a}szl{\'o} Lov{\'a}sz.
\newblock {\em Large Networks and Graph Limits}, volume~60 of {\em Colloquium
  Publications}.
\newblock American Mathematical Society, Providence, 2012.

\bibitem{LS:2006}
L{\'a}szl{\'o} Lov{\'a}sz and Bal{\'a}zs Szegedy.
\newblock Limits of dense graph sequences.
\newblock {\em J. Combin. Theory Ser. B}, 96(6):933--957, 2006.

\bibitem{LS:2007}
L{\'a}szl{\'o} Lov{\'a}sz and Bal{\'a}zs Szegedy.
\newblock Szemer\'edi's lemma for the analyst.
\newblock {\em Geom. Funct. Anal.}, 17(1):252--270, 2007.

\bibitem{LS:2010}
L{\'a}szl{\'o} Lov{\'a}sz and Bal{\'a}zs Szegedy.
\newblock Testing properties of graphs and functions.
\newblock {\em Israel J. Math.}, 178:113--156, 2010.

\bibitem{LS:2011}
L{\'a}szl{\'o} Lov{\'a}sz and Bal{\'a}zs Szegedy.
\newblock Finitely forcible graphons.
\newblock {\em J. Combin. Theory Ser. B}, 101(5):269--301, 2011.

\bibitem{RCY:2011}
Karl Rohe, Sourav Chatterjee, and Bin Yu.
\newblock Spectral clustering and the high-dimensional stochastic blockmodel.
\newblock {\em Ann. Statist.}, 39(4):1878--1915, 2011.

\bibitem{Rudin:1964}
W. Rudin.
\newblock {\em Principles of mathematical analysis}.
\newblock Second edition. McGraw-Hill Book Co., New York, 1964.

\bibitem{Sidorenko:1993}
Alexander Sidorenko.
\newblock A correlation inequality for bipartite graphs.
\newblock {\em Graphs Combin.}, 9(2):201--204, 1993.

\bibitem{Troutman:1983}
John~L. Troutman.
\newblock {\em Variational calculus with elementary convexity}.
\newblock Undergraduate Texts in Mathematics. Springer-Verlag, New York, 1983.
\newblock With the assistance of W. Hrusa.

\bibitem{Vershik:1998}
Anatolii~Moiseevich Vershik.
\newblock The universal {U}rysohn space, {G}romov metric triples and random
  metrics on the natural numbers.
\newblock {\em Russian Mathematical Surveys}, 53(5):921--928, 1998.

\bibitem{Whitney}
Hassler Whitney.
\newblock The coloring of graphs.
\newblock {\em Ann. of Math. (2)}, 33(4):688--718, 1932.
\end{thebibliography}
\end{document}